\documentclass[11pt, reqno]{amsart}
\usepackage[T1]{fontenc}
\usepackage[english]{babel}

\usepackage{amsmath}
\usepackage{biblatex}
\usepackage{amssymb}
\usepackage{amsthm}
\usepackage{bbm}
\usepackage{mathrsfs}
\usepackage{epsfig}
\usepackage{enumerate}
\usepackage{mathtools}
\usepackage{graphicx}
\usepackage[usenames,dvipsnames]{xcolor}
\usepackage{appendix}
\usepackage[font=small]{caption}
\usepackage[colorlinks=true,linkcolor=NavyBlue,urlcolor=RoyalBlue,citecolor=PineGreen,bookmarks=true,bookmarksopen=true,bookmarksopenlevel=2,unicode=true,linktocpage]{hyperref}
 \usepackage[all]{xy}

\newtheorem{theoremi}{Theorem}[]

\newtheorem{theorem}{Theorem}[section]
\newtheorem{lemma}[theorem]{Lemma}

\newtheorem{proposition}[theorem]{Proposition}
\newtheorem*{proposition*}{Proposition}
\newtheorem*{lemma*}{Lemma}
\newtheorem*{theorem*}{Theorem}
\newtheorem{corollary}[theorem]{Corollary}
\newtheorem{definition}[theorem]{Definition}
\newtheorem{remark}[theorem]{Remark}

\DeclareMathOperator{\Aut}{Aut}
\DeclareMathOperator{\sgn}{sgn}

\DeclareMathOperator{\ad}{ad}
\DeclareMathOperator{\Ad}{Ad}
\DeclareMathOperator{\id}{id}

\def\d{\;{\rm d}}

\def\R{\mathbb{R}}

\def\C2{{\sf{C}_2}}

\def\epsilon{\varepsilon}

\def\build\#1_\#2^\#3{\mathrel{\mathop{\kern 0pt\#1}\limits_{\#2}^{\#3}}}

\def\chint\#1{\mathchoice
   {\XXint\displaystyle\textstyle{\#1}}%
   {\XXint\textstyle\scriptstyle{\#1}}%
   {\XXint\scriptstyle\scriptscriptstyle{\#1}}%
   {\XXint\scriptscriptstyle\scriptscriptstyle{\#1}}%
   \!\int}
\def\XXint\#1\#2\#3{{\setbox0=\hbox{$\#1{\#2\#3}{\int}$}
     \vcenter{\hbox{$\#2\#3$}}\kern-.5\wd0}}

\title[Brownian homotopy around a Poisson cloud]{Homotopy and holonomy of the planar Brownian motion in a Poisson punctured plane}
\author{Isao Sauzedde}

\address{Isao Sauzedde -- Oxford University}
\email{isao.sauzedde@maths.ox.ac.uk}


\begin{document}
\begin{abstract}
We define a family of diffeomorphism-invariant models of random connections on principal $G$-bundles over the plane, whose curvatures are concentrated on singular points.
In a limit when the number of point grows whilst the singular curvature on each point diminishes, the model converges in some sense towards a Yang--Mills field. We study another regime for which we prove that the holonomy along a Brownian trajectory converges towards an explicit limit.
\end{abstract}

\maketitle
\tableofcontents

\section{Introduction}


This paper is concerned with the way in which the Brownian motion winds around points in the plane. This question was studied in the past mostly from a homological point of view, that is by studying the winding of the Brownian motion around the points. One can mention in particular the celebrated result of Spitzer \cite{Spitzer} about the Cauchy behaviour of the asymptotic winding around one point, when the Brownian motion runs up to a time that tends to infinity. It paved the way for an exact computation by Yor \cite{Yor} of the distribution of this winding, at a fixed time, and also for further limits theorems when time tends to infinity \cite{Pitman}. Asymptotic windings around a family of points are also well known from the homological perspective (see e.g. \cite{MansuyYor}). 
Following ideas of Werner on the set of points with given winding \cite{Werner2,Werner}, we showed in a previous paper \cite{LAWA} that, when the points are distributed according to a Poisson process $\mathcal{P}$ whose intensity goes to infinity, almost surely with respect to the Brownian motion, the averaged winding converges in distribution towards a Cauchy variable centered at the Lévy area of the process.

In this paper, we try to go beyond the homological point of view and capture some homotopical information. This is done by defining a random flat $G$-bundle over $\R^2\setminus \mathcal{P}$, each point being associated with a curvature which is scaled to grow as the inverse of the intensity of the process. We then  study the holonomy of the Brownian motion and show that, for any connected compact Lie group $G$, as the intensity of the point process goes to infinity, the Brownian holonomy converges, almost surely on the Brownian trajectory and in distribution on $\mathcal{P}$, towards the analogue in $G$ of a Cauchy distribution.
\\

The model is also motivated by its relation with the Yang--Mills field \cite{Driver}. The $2$-dimensional Yang--Mills field over the plane is a random map from a set of sufficiently smooth loops on $\R^2$ to a compact connected Lie group $G$. This map satisfies the algebraic properties of the holonomy function associated to a smooth connection on the trivial $G$-bundle over $\R^2$, but it is much less regular. The Yang--Mills field is therefore heuristically understood as the holonomy map of a very irregular random connection on this bundle. In the Sobolev scale, the regularity of this random connection is expected to be $H^{-\epsilon}$, and this is too low to allow for a reasonable direct extension of the holonomy to typical trajectories of the Brownian motion.


In the particular case where the group $G$ is abelian with Lie algebra $\mathfrak{g}$, the Yang--Mills field can be built from a $\mathfrak{g}$-valued white noise $\Phi$ on $\R^2$. For a loop $\ell$ based at $0$ and that is smooth enough, the Yang--Mills holonomy alongs $\ell$ is then simply given by $\Phi(\theta_\ell)$, where $\theta_\ell$ is the winding function associated with $\ell$. This of course requires $\ell$ to be regular enough for $\theta_\ell$ to be square integrable, which is the case when $\ell\in \mathcal{C}^1$.


Another way to define a map from a set of loops based on $0$ to $G$ in an multiplicative way is to take a  Poisson process $\mathcal{P}$, and then to define a group homomorphism $h$ from the fundamental group $\pi_1(\R^2\setminus \mathcal{P},0)$ to $G$.
We will define a family of such random pairs $(\mathcal{P},h)$, indexed by two parameters $K$ (the intensity of the point process) and $\iota$ (the curvature or charge associated with each point) in such a way that their law is invariant under volume-preserving diffeomorphisms and gauge transformations. When $\iota$ is proportional to $K^2$ and $K$ tends to infinity, these random connections converge towards a Yang--Mills field, in the sense of finite-dimensional marginals. Our goal, however, is to study \emph{another} scaling regime, when $\iota$ is proportional to $K$ instead. Then, the holonomy of any given smooth loop converges in distribution towards $1_G$, but the holonomy along the Brownian curve (made into a loop by joining the endpoints smoothly) does not. Instead, it converges towards a $1$-stable distribution in $G$.

One should mention that a related question was already studied long ago, in the commutative case, by Albeverio and Kusuoka \cite{Albeverio} (see the main theorem there), with a radically different approach.
Instead of rescaling the field, they embed the group in a Euclidean space, average the holonomy, and then rescale this average.

\section{Definition of the model and presentation of the main result}

Let $G$ be a connected compact Lie group, of which we denote the unit element by $1$. We endow its Lie algebra $\mathfrak{g}$ with a bi-invariant scalar product, and we denote by $\|\cdot \|$ the associated norm. This Euclidean structure on $\mathfrak g$ gives rise to a bi-invariant Riemannian structure on $G$, and we denote by $d_G$ the corresponding distance. For a positive real number $K$, the uniform probability distribution on the sphere of radius of $\frac{1}{K}$ of $\mathfrak{g}$ is denoted $\mu_K$, and for $K=1$ we set $\mu=\mu_1$.

%


Let then
 $\mathcal{P}^\mathfrak{g}_K$ be a Poisson process on $\R^2\times \mathfrak{g}$, with intensity $ K \lambda\otimes \mu_K$, where $\lambda$ is the Lebesgue measure on $\R^2$. Let also
 \[ \mathcal{P}^G_K\coloneqq\{(x,\exp(Z)): (x,Z)\in  \mathcal{P}^\mathfrak{g}_K\}\]
and  $\mathcal{P}_K$ be the projection of $\mathcal{P}^G_K$ on $\R^2$. The projection $\mathcal P^G_K\to \mathcal P_K$ is almost surely injective, so that there exists a map $h:\mathcal P_K\to G$ such that
\[\mathcal P^G_K=\{(x,h(x)) : x\in \mathcal P_K\}.\]
Once we fix  an isomorphism $\Phi$ from the fundamental group $\pi_1(\mathcal{P}_K)\coloneqq \pi_1(\mathbb{R}^2\setminus \mathcal{P}_K, 0)$ to the free group $\mathbb{F}_{ \mathcal{P}_K}$ generated by $\mathcal{P}_K$, the map $h$ induces a flat connection $\omega_{\Phi,K}$, up to gauge equivalence, on the bundle $(\mathbb{R}^2\setminus \mathcal{P}_K) \times G$ over $\mathbb{R}^2\setminus \mathcal{P}_K$. We will apply this to a specific family of isomorphisms $\Phi$ that we will describe in Definition \ref{def:proper}, and call \emph{proper}.

\medskip
 Let also $X:[0,1]\to \R^2$ be a planar Brownian motion started from $0$, independent from $\mathcal{P}^G_K$, and defined on a probability space $(\Omega^X,\mathcal{F}^X,\mathbb{P}^X)$, and let $\bar{X}$ be its concatenation with the straight line segment between its endpoints.
Then, $\bar{X}$ almost surely avoids $\mathcal{P}_K $, so that the holonomy $\omega_{\Phi,K} (\bar{X})$ of $\omega_{\Phi,K}$ along $\bar{X}$ is well defined.

The main result of this paper is the following.
\begin{theoremi}
\label{th:main}
Assume that the proper isomorphism $\Phi$ between $\pi_1(\mathcal{P}_K)$ and $\mathbb{F}_{ \mathcal{P}_K}$ is independent of $(X,\mathcal{P}^G_K)$ conditional on $\mathcal{P}_K$.

Then, $\mathbb{P}^X$-almost surely, as $K\to \infty$, $\omega_{\Phi,K}(\bar{X})$ converges in distribution towards a $1$-stable distribution on $G$ that does not depend on $\Phi$.
\end{theoremi}

The limiting distribution $\nu^*$ that we call $1$-stable\footnote{The terminology \emph{stable-like} is sometimes preferred, because such laws lack a nice scaling property when the group is not abelian.} can be described as follows. Let $d$ denote the dimension of $G$ and
\[\sigma= \frac{\Gamma\big(\tfrac{d}{2})}{ 2 \sqrt{\pi} \Gamma\big(\tfrac{d+1}{2}\big)}.\]. Let $\nu^{\sigma}$ be the symmetric $1$-stable probability distribution on $\mathfrak{g}$ defined by its density with respect to the Lebesgue measure
\[
\frac{\d \nu^{\sigma}(Z)}{\d Z}=\frac{1}{C_1}\frac{1}{\sigma^d (1+\sigma^{-2}\|Z\|^2)^{\frac{d+1}{2}} }, \qquad C_1=\int_{\mathfrak{g}} \frac{1}{\sigma^d (1+\sigma^{-2}\|Z\|^2)^{\frac{d+1}{2}} }\d Z=\frac{\pi^{\frac{d+1}{2}} }{\Gamma(\frac{d+1}{2}) } .
\]
There exists then a $\mathfrak g$-valued symmetric $1$-stable process $\tilde{Y}$ such that $\tilde{Y}(1)$ is distributed according to $\nu^{\sigma}$. This process has jumps, but we can `fill' them with straight-line segments into a continuous curve $Y:[0,1]\to \mathfrak{g}$, defined up to increasing parametrization of $[0,1]$. Such a parametrization can \cite[Theorem 4.1]{Blumenthal} be chosen in such a way that $Y$ has finite $p$-variation for $p>1$. In particular, taking $p<2$ allows to define the Cartan development $y$ of $Y$ on $G$.
Then, $\nu^*$ is equal to the distribution of $y(1)$.

\subsection{Strategy of the proof }
In order to guide us through the proof, we first give an idea of why the main theorem holds, and of the main difficulty that arises, before giving a summary of the steps involved to solve it.

When a proper isomorphism has been fixed, the homotopy class $[\ell]\in \pi_1(\mathcal{P}_K)$ of a loop $\ell$ can be understood as a word on the letters $x,x^{-1}: x\in \mathcal{P}_K$. If we count algebraically each apparition of a letter $x$ in $[\ell]$, the relative integer that we obtain is equal to the winding number $\theta(x,\ell)$ of $\ell$ around the point $x$. If the order in which these letters appear were not important, the class $[\ell]$ would thus be equal to
\[x_1^{\theta(x_1,\ell)}\dots x_n^{\theta(x_n,\ell)},\]
where $\{x_1,\dots, x_n\}=\mathcal{P}_K$. In Section \ref{sec:simpler}, we will prove that Theorem  \ref{th:main} does hold if we replace $\omega_{\Phi,K}(\bar{X})$
with
\[g_1^{\theta(x_1,\ell)}\dots g_n^{\theta(x_n,\ell)},\]
where  $\{(x_1,g_1),\dots, (x_n,g_n)\}=\mathcal{P}_K^G$ is an ordering of $\mathcal{P}_K^G$, independent of $\mathcal{P}_K^G$ conditionally to $\mathcal{P}_K$. This will follow rather easily from the corresponding result with $G$ replaced with $\mathbb{R}$, and indeed proves the theorem in the case the group $G$ is abelian.

Let us now explain why it is reasonable to think that the order doesn't matter indeed.
First of all, with the kind of scaling that we have, and with the heavily tailed Cauchy distribution appearing, only the extremely large exponents (of the order of $K$) should contribute to the limit.
If a point appears once with an exponent of the order of $K$, it likely has a winding number which is at least of the order of $K$.
The number of such points in $\mathcal{P}$ is only of the order of $1$. Let us consider $x$ one of them. Then, the Brownian motion $X$ must go extremely close to $x$. Let us say that $[s,t]$ is some small amount of time such that $X_s$ and $X_t$ are relatively far from $x$, but during which $X$ goes very close to $x$. During that time, it might wind a large amount of time, say $\hat{\theta}(x)$, around $x$. Then, it is very unlikely that there is a \emph{second} period of time during which $X$ goes close to $x$, and since this is the only reasonable way for $X$ to accumulate windings around $x$, we can deduce that $\theta(x)\simeq \hat{\theta}(x)$.

During the same period $[s,t]$, the homotopy class of $X$ will vary a lot, it will start from say $g$ to $g h x^{\hat{\theta}(x)} h'$.
\footnote{To be slightly more specific, we are comparing here the homotopy class of $X_{|[0,s]}\cdot [X_s,X_0]$ with the homotopy class of $X_{|[0,t]}\cdot [X_t,X_0]$.}
Actually, it is possible to choose $s$ and $t$ so that
additionally $h=h'=1$, provided that the proper isomorphism is chosen so that $x$ corresponds to the path going straight ahead from $0$ to $x$, winding once around $x$, and then going back straight to $0$. If $x_1,\dots, x_k$ are these special points that appears once with a large exponent, we end up with a writing of $[\bar{X}]$ as
\begin{align*}
[\bar{X}]&=[X_{0,s_1}] x_1^{\hat{\theta}(x_1)} [X_{t_1,s_2}] x_2^{\hat{\theta}(x_2)} \dots, [X_{t_k,1}]\\
&= \big([X_{0,s_1}]x_1^{\hat{\theta}(x_1)}[X_{0,s_1}]^{-1}\big) \dots \big([X_{0,s_1}][X_{t_1,s_2}]\dots [X_{t_{k-1}},s_{k}]
x_k^{\hat{\theta}(x_k)}[X_{t_{k-1}},s_{k}]^{-1}\dots [X_{0,s_1}]^{-1}\big) \\
&\hspace{10cm} [X_{0,s_1}]\dots [X_{t_{k-1}},s_{k}] [X_{t_k,1}]
\end{align*}
where $X_{s,t}$ is the loop that start from $0$, goes straight to $X_s$, then follows $X$ up to $X_t$, and then goes straight to $0$. Since the loops $X_{t_i,s_{i+1}}$ do not have large winding around any point, it is expected that $ \omega_{\phi,K}([X_{0,s_1}]\dots [X_{t_k,1}])$ converges towards $0$ as $K\to \infty$. From the conjugation invariance of the variable associated with $x_1,\dots, x_k$, one can then expect that the distribution of $\omega_{\phi,K}([\bar{X}])$ is very close from the one of $g_1^{\hat{\theta}(x_1)}\dots g_k^{\hat{\theta}(x_k)}$, which itself is very close from $g_1^{\theta(x_1)}\dots g_n^{\theta(x_n)}$.

One of the main problems we face when trying to make this rigorous is to actually get some control on $\omega_{\phi,K}([X_{0,s_1}]\dots [X_{t_k,1}])$. The problem is that the length of the word corresponding to $[X_{0,s_1}]\dots [X_{t_k,1}]$ cannot be easily bounded, and we expect it to be way too large in general for a naive approach to be conclusive. What happens is that each time the Brownian motion does a `large winding' around $x\in \mathcal{P}$ (that is a winding during which it is not specifically close from $x$), what will appear in $[X_{t_{i},s_{i+1}}]$ is not only $x^{\pm 1}$, but instead a conjugate $c x^{\pm 1}c^{-1}$, where $c$ is a word whose length is of order $K$ in general. Though the variables corresponding to $x$ and $cxc^{-1}$ are identical in law, replacing one with the other would break the independence between the variables --as soon as $X$ does a second turn around $x$, we get both $cxc^{-1}$ and $c'x{c'}^{-1}$ appearing in $[X_{0,s_1}]\dots [X_{t_k,1}]$. This independence is needed to apply some law of large numbers and show that $\omega_{\phi,K}([X_{0,s_1}]\dots [X_{t_k,1}])$ converges to zero. We solve this problem by finding a hierarchical structure between the different points that allows simultaneously to keep some independence and to ignore the size of the conjugators $c$.


To be more specific, we will order the points in $\mathcal{P}=\{x_1,x_2,\dots\}$ and decompose the class $h=[\bar{X}]$ into a product
\[
h=h_1\dots h_n,
\]
and decompose further each $h_i$ as
\[h_i=h_{i,1}\dots h_{i,j_i}\]
in such a way that each $h_{i,j}$ is a conjugate of $x_i^{\pm 1}$, say $h_{i,j}=c_{i,j}x_i^{\pm 1} c_{i,j}^{-1}$, with only the letters $(x_k)_{k<i}$ allowed to appear in $c_{i,j}$.
The goal of Section~\ref{sec:free} is to show that such a decomposition is always possible, and to understand some of its aspects. 

Its only on Section \ref{sec:tech}, with Proposition \ref{prop:tech}, that we will study the probabilistic interest of such a decomposition. To each $x_i$ will be associated a random matrix $X_i\simeq 1$, and we will study the product $X_h=X_{i_1}\dots X_{i_k}$, where $h=x_{i_1}\dots x_{i_k}$.
Roughly speaking, we then manage to control the distance from $X_h$ to the identity in a way that depends only on the $j_i$, and with a bound which is almost the same as when the group is commutative.

The {geometric} interest of this decomposition is studied in Section \ref{sec:relations}. Here, we make a specific choice for the order between the points in $\mathcal{P}$. We order them depending on their distance to the origin. This choice allows to show that for each point $x_i$, the number $j_i$ cannot be larger than some number $\theta_{\frac{1}{2}}(x_i)$  which, roughly speaking, counts some number of half-turns of $X$ around $x_i$. The important thing is that this number $\theta_{\frac{1}{2}}(x_i)$ depends on $X$ and on $x_i$, but not on the other points of $\mathcal{P}$.

In Section \ref{sec:brown}, we finally use the fact that our path is Brownian. We will see that the quantity $\theta_{\frac{1}{2}}(x_i)$ turns out to be strongly related to the winding number $\theta(x_i)$, and to be roughly speaking of the order of $|\theta(x_i)|^2$. We will also consider the partition $\beta_{i,l}$ of $j_i$, defined by the indices $k$ such that the conjugators $c_{i,k}$ and $c_{i,k+1}$ are different. We will give various upper bounds (with large probability) on the highest values $\beta_{i,1},\beta_{i,2},\beta_{i,3},\dots$ and on the possible size of set of indices $i$ for which these values can be large.

In Section \ref{sec:end}, we will arrange these bounds together to finally prove the main theorem.

We did not yet mention Section \ref{sec:inv}. It is devoted to show that the distribution of the flat connection we define does not depend on the choice of the proper isomorphism, and that it is invariant by volume-preserving diffeomorphisms of $\R^2$.

\section{Invariance by volume-preserving diffeomorphism: from a point set to a connection}
In this section, we define a family of isomorphisms $\Phi$ between $\pi_1(\mathcal{P})$ and the free group $\mathbb{F}_{\mathcal{P}}$ with $\mathcal{P}$ generators. This family does not depend on any additional data, such as a Riemannian structure on the plane. We then show that the random connection $\omega_{\Phi,K}$ defined in the previous section does not depend, in distribution, from the isomorphism $\Phi$ chosen in this family. We will see that two of them are always related by the action of a braid, so that we should first understand how braids act on the distribution of random variables.

 \label{sec:inv}
\subsection{An easily lemma}
Let $B_n$ be the Artin braid group with $n$ strands,
\[ B_n=\langle b_1,\dots, b_{n-1}|  \forall i\leq n-2 : b_ib_{i+1}b_i=b_{i+1}b_i b_{i+1};  \forall i,j\leq n-1, |j-i|\geq 2: b_ib_j=b_jb_i \rangle.\]
It admits a group homomorphism onto the symmetric group $S_n$, which maps the generator $b_i$ to the permutation $\sigma(b_i)=\sigma_i=(i,i+1)$. We denote by $\sigma$ this projection.
It also admits a natural action on $G^n$, given by
\[ b_i\cdot(g_1,\dots, g_n)= (g_1,\dots, g_{i-1}, g_{i+1}, g_{i+1}^{-1}g_i g_{i+1}, \dots , g_n).\]

We propose here a simple lemma, which we do not use directly but is somehow the key in the proof of the volume-preserving diffeomorphism invariance that we prove.
\begin{lemma}
  Let $(X_1,\dots, X_n)$ be a family of independent $G$-valued random variables, each conjugation invariant. Then, the following equalities in distribution hold:
  \begin{itemize}
  \item for all $i\in \{1,\dots, n-1\}$,
  \[
  b_i\cdot(X_1,\dots, X_n)\overset{(d)}=(X_{\sigma_i(1)}, \dots, X_{\sigma_i(n)}).
  \]
  \item More generally, for all $b\in B_n$,
  \[
  b\cdot(X_1,\dots, X_n)\overset{(d)}=(X_{\sigma(b)(1)}, \dots, X_{\sigma(b)(n)}).
   \]
  \end{itemize}
\end{lemma}
\begin{proof}
  For the first item, we need to prove is that for all $i\in \{1,\dots, n-1\}$,
  \[(X_1,\dots, X_{i-1},X_{i+1}, X_{i+1}^{-1}X_i X_{i+1}, X_{i+2},\dots, X_n )\overset{(d)}= (X_1,\dots X_{i-1}, X_{i+1},X_{i}, X_{i+2}, \dots X_n).
  \]
  From the independence assumption, this reduces to show that
  \[
  (X_{i+1}, X_{i+1}^{-1}X_i X_{i+1})\overset{(d)}= (X_{i+1},X_i).
  \]
  For $j\in \{i,i+1\}$, let $\mathbb{P}^j$ be the law of $X_j$. Then, for all bounded $\mathbb{P}^i\otimes \mathbb{P}^{i+1}$-measurable function $f$ from $G^2$ to $\R$,
  \begin{align*}
  \mathbb{E}[f(X_i,X_{i+1} )]&=
  \int_G \Big( \int_G f(u,v) \d \mathbb{P}^{i}_u \Big) \d \mathbb{P}^{i+1}_v\\
  &= \int_G \Big( \int_G f(v^{-1}uv,v) \d \mathbb{P}^{i}_u \Big) \d \mathbb{P}^{i+1}_v \qquad \mbox{(from conjugation invariance of $\mathbb{P}^i$)}\\
  &=
  \mathbb{E}[f(X_{i+1}^{-1}X_iX_{i+1}, X_{i+1})],
  \end{align*}
  which concludes the proof. Remark that we do not need the conjugation invariance of $X_n$ here.

  For a general $b\in B_n$, it suffices to write $b=b_{i_1}^{\epsilon_1}\dots b_{i_n}^{\epsilon_n}$, with the $\epsilon_i$ in $\{\pm 1\}$, and use a recursion. It is easily seen that $(X_{\sigma(1)}, \dots, X_{\sigma(n)} )$ satisfies the same assumptions as $(X_1,\dots X_n)$ (for which we do need the conjugation invariance of $X_n$), which allows to apply the first of the lemma and conclude the recursion.
\end{proof}
\subsection{Finite set of punctures}

The plane is endowed with its differential structure and orientation, and a point $0$ is fixed.
Let $\mathcal{P}$ be a finite subset of $\R^2 \setminus\{0\}$.

The fundamental group $\pi_1(\mathcal{P})=\pi_1(\R^2\setminus\mathcal{P},0)$ is a free group with $\# \mathcal{P}$ generators. A basis can be obtained by choosing $\#  \mathcal{P}$ non-intersecting and non self-intersecting smooth paths $(\gamma_x)_{x\in \mathcal{P}}$, where $\gamma_x$ is a path from $0$ to $x$ in $K$. Any such path $\gamma_x$ then determines a unique element $\ell_x$ in $\pi_1(\mathcal{P} )$, a representative of which is given by a path that goes from $0$ to a small neighbourhood of $x$ following $\gamma_x$, then turns once around $x$
in trigonometric order,
and finally goes back to $0$ following $\gamma$ backward.
The family $(\ell_x)_{x\in \mathcal{P}}$ then freely generates $\pi_1(\R^2\setminus\mathcal{P},0)$. Figure \ref{fig:isom} below represents a possible choice for the $\gamma_x$, and a path in one of the corresponding classes $\ell_x$. We will call \emph{proper} any bases obtained that way.
\begin{figure}[h!]
\includegraphics{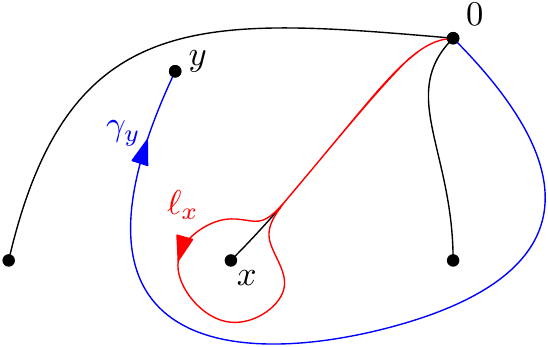}
\caption{\label{fig:isom} An example for $\# \mathcal{P}=4$.}
\end{figure}

We insist on the fact that the proper bases are not indexed by $\{1,\dots, \# \mathcal{P} \}$, but by $\mathcal{P}$.

A proper basis $\mathfrak{b}=([\ell_x])_{x\in \mathcal{P}}$ together with a family $(\omega_x)_{x\in \mathcal{P}}$ of elements of $G$ determine a unique group morphism $\omega^{\mathfrak{b}}$ from $\pi_1(\mathcal{P})$ to $G$ which is such that $\omega^{\mathfrak{b}}([\ell_x])=\omega_x$.
In order to understand how $\omega^{\mathfrak{b}}$ depends on $\mathfrak{b}$, one should first understand how proper bases are related one to the other. To this end, let us define $\mathfrak{B}$ the set of proper bases, and $\Aut$ the set of orientation-preserving homeomorphisms of $\R^2\setminus \mathcal{P}$ that maps $0$ to itself. For $\mathfrak{b}\in \mathfrak{B}$ associated with the paths $\gamma_x$, and $\phi\in \Aut$, let $\phi(\mathfrak{b})$ be the proper basis associated with the paths $\phi(\gamma_x)$.
\begin{lemma}
  The application $(\phi,\mathfrak{b})\mapsto \phi(\mathfrak{b})$ from $\Aut\!\times \mathfrak{B}$ to $\mathfrak{B}$ defines a transitive group action.
\end{lemma}
\begin{proof}
  The fact that it is a group action is direct, only the transitivity needs to be shown. Figure \ref{fig:diffeo} below pictures the different steps of the proof.
  \begin{figure}[h]
  \includegraphics{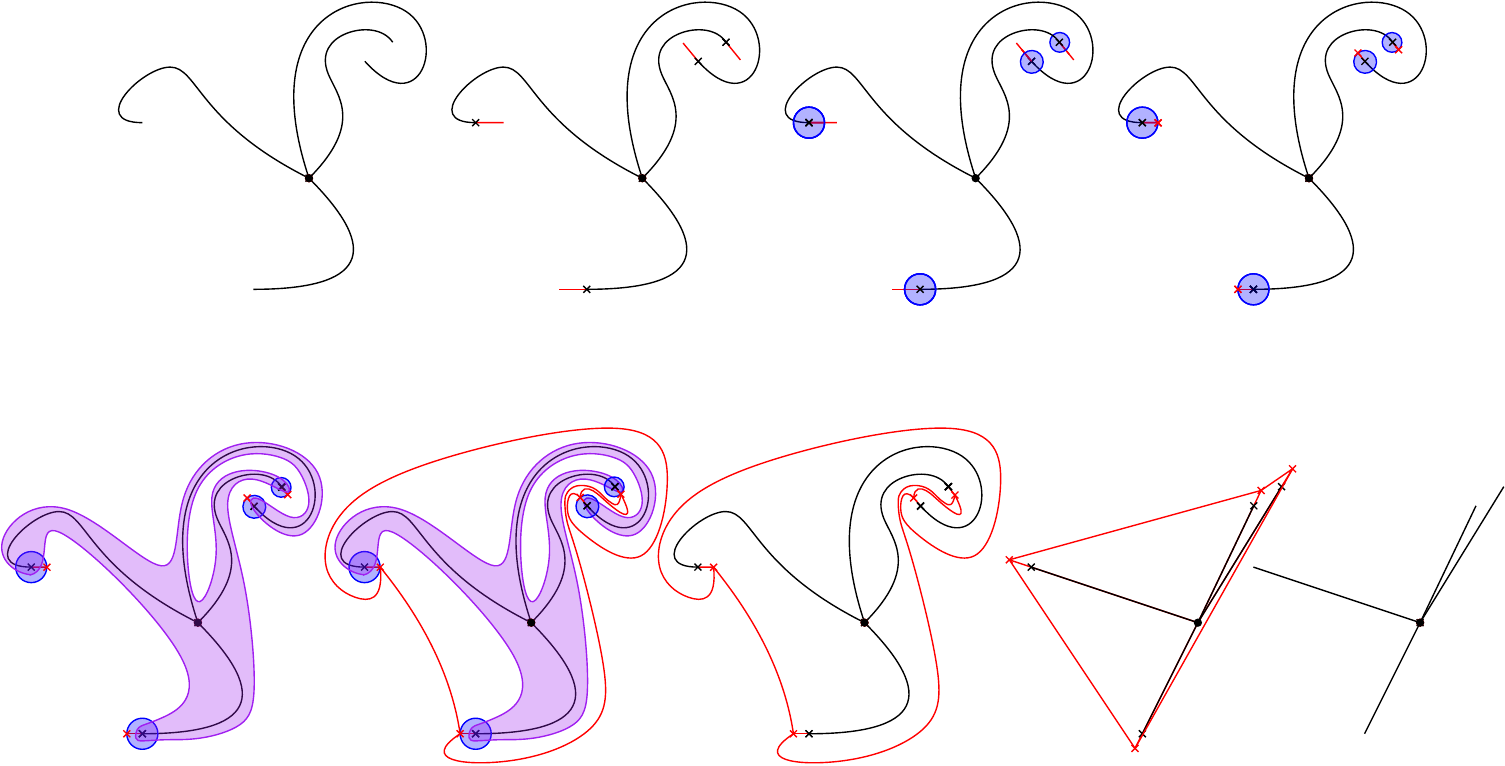}
  \caption{\label{fig:diffeo} Black lines: the different paths $\gamma_x$. Red lines: the geodesic continuations $\gamma'_x$. Blue circles: The neighbourhoods $U_x$. Red crosses: the points $\tilde{x}$. Purple set: the set $U_\infty$. Red curves: the curves $\tilde{\gamma}_x$. Eight picture: image of the previous one after application of an homeomorphism on each of the connected components.}
  \end{figure}

  We fix some Riemannian metric on $\R^2$. Then, for each $x\in \mathcal{P}$, let $\gamma'_x$ be the geodesic of length $\epsilon$ from $x$ with the same tangent vector as $\gamma'_x$ at $x$. For $\epsilon>0$ small enough, this geodesic does not cross any of the path $\gamma_x$ nor any other $\gamma'_y$.
  Let $U_x$ be small neighbourhoods of the points $x$, diffeomorphic to open balls, two by two disjoints, and small enough for $\gamma'_x$ to reach the boundary of $U_x$.
  For each $x$, let $t_x$ be the first time $\gamma'_x$ reaches the boundary of $U_x$. Let then $\tilde{x}=\gamma'_{x}(t_x)$ and $\gamma''(x)$ the restriction of $\gamma'(x)$ to $[0,t_x]$.

  Let then $\Gamma$ be a relatively compact neighbourhood of
  the union of the ranges of the $\gamma_x$, with a smooth boundary, and such that each of the $\tilde{x}$ lies the unbounded component delimited by $\Gamma$. Set $U_\infty=\R^2\setminus (\Gamma \cup \bigcup_{x\in \mathcal{P}} U_x)$.

  The boundary of $U_\infty$ is smooth by part, thus homeomorphic to a circle, which induces a cyclic order between the points $\tilde{x}$. Since $K_\infty$ is homeomorphic to $\R^2\setminus B(0,1)$, one can find for each $x$ a continuous path $\tilde{\gamma}_x$ from $\tilde{x}$ to its successor $s(\tilde{x})=\widetilde{s(x)}$, in such a way that these paths stay on $K_\infty$ and do not intersect each other.

  Then, for each $x$, the concatenation of $\gamma_x$, $\gamma''_x$, $\tilde{\gamma}_x$ and $\gamma_{s(x)}^{-1}$ is a continuous loop, and Jordan theorem ensures the existence of an homeomorphism from its interior to an open triangle. Applying this to each of these loop, and to the unbounded component delimited by the concatenation $\tilde{\gamma}_x \tilde{\gamma}_{s(x)} \tilde{\gamma}_{s^2(x)}\dots$, we obtain the desired homeomorphism.
\end{proof}



\begin{lemma}
Let $\Aut_0$ be the set of elements of $\Aut$ isotopic to the identity. Then, $\Aut_0$ acts trivially on $\mathfrak{B}$. That is, for every $\phi\in \Aut_0$ and $\mathfrak{b}\in \mathfrak{B}$, $\phi(\mathfrak{b})=\mathfrak{b}$. Stated otherwise, the action of $\Aut$ on $\mathfrak{B}$ passes to the quotient into a transitive action of the mapping class group $\pi_0(\Aut)=\Aut/\Aut_0$.
\end{lemma}
\begin{proof}
Let $\phi\in \Aut_0$, and $h$ be an isotopy from $\phi$ to the identity. Then, for all $x\in \mathcal{P}$, $\gamma_x=(\gamma_{x,s})_{s\in S^1}$ is homotopic to $\phi(\gamma_x)$ by
$(s,t)\mapsto h_t(\gamma_{x,s})$. Therefore, $[\gamma_x]=[\phi(\gamma_x)]$ and $\mathfrak{b}=\phi(\mathfrak{b})$.
\end{proof}
This allows us to deduce the following result.
\begin{lemma}
\label{le:invariance}
Let $\mathcal{P}$ be a finite set and $(\omega_x)_{x\in \mathcal{P}}$ be a family of independent $G$-valued random variable. Assume that each of the variables $\omega_x$ has a distribution which is invariant by conjugation. Then, the law of $\omega^{\mathfrak{b}}$ does not depend on the choice of the proper basis $\mathfrak{b}\in \mathfrak{B}$.
\end{lemma}
\begin{proof}
Let $n=\# \mathcal{P}$. The mapping class group $\pi_0(\Aut)$ of the $n^{\mbox{\scriptsize{th}}}$ punctured plane is known to be isomorphic to the Artin braid group $B_n$ with $n$ strands (see e.g. \cite{Artin,braids}),
If we enumerate $\mathcal{P}=\{x_1,\dots, x_n\}$, then the isomorphism can be chosen such that the preimage of $b_i$ is given by the classes $c_{i}\in \pi_0(\Aut)$ which maps $[\ell_{x_j}]$ to itself for $j\notin \{i,i+1\}$, to $[\ell_{x_i+1}]$ if $j=i$, and to $[\ell_{x_{i+1}} \ell_{x_i}\ell_{x_{i+1}}^{-1} ]$ if $j=i+1$.
The action of such a braid on $\mathfrak{B}$ is pictured in Figure \ref{fig:homot} below.
\begin{figure}[h!]
\includegraphics{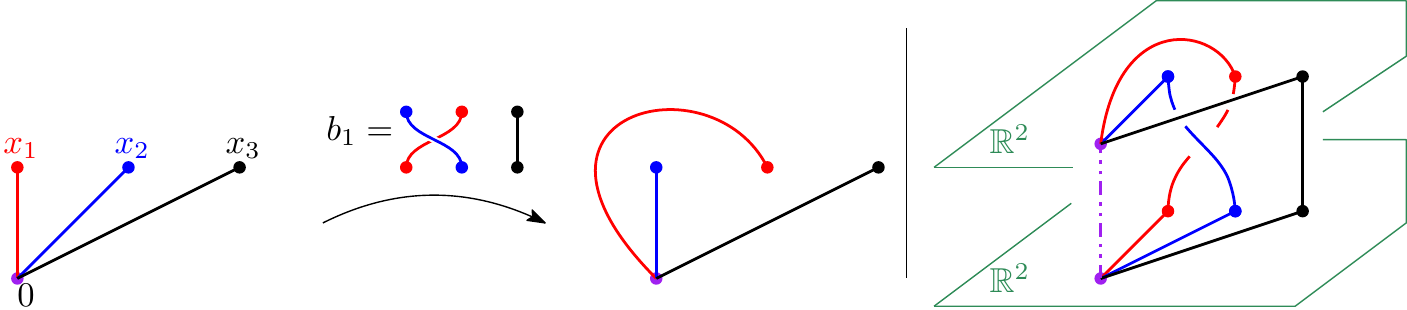}
\caption{\label{fig:homot} Two depictions of the braid $b_{1}$ acting on a proper basis.}
\end{figure}

Since the $b_i$ generates $B_n$, the $c_i$ generates $\pi_0(\Aut)$. Since the action of $\pi_0(\Aut)$ on $\mathfrak{B}$ is transitive, it suffices to show that for some proper basis $\mathfrak{b}=([\ell_x])_{x\in \mathcal{P}}$,
for all $b\in B_n$, $\omega^{\mathfrak{b}}$ and $\omega^{b(\mathfrak{b})}$ are equal in distribution. Since $\pi_1(\mathcal{P})$ is generated by the $\ell_x$ and both
$\omega^{\mathfrak{b}}$ and $\omega^{b(\mathfrak{b})}$ are group morphisms, it suffices to show
\begin{equation}
\label{eq:star}
\big(\omega^{\mathfrak{b}} ([\ell_{x_1}]),\dots, \omega^{\mathfrak{b}} ([\ell_{x_n }])\big)\overset{(d)}= \big(\omega^{b(\mathfrak{b})} ([\ell_{x_1}]),\dots, \omega^{b(\mathfrak{b})} ([\ell_{x_n }])\big).
\end{equation}
We show by recursion on $k$ that for \emph{all} proper basis $\mathfrak{b}=([\ell_x])_{x\in \mathcal{P}}$, for all $b=b^{\epsilon_1}_{i_1}\dots b^{\epsilon_k}_{i_k}\in B_n$, with the $\epsilon_i$ in $\{\pm 1\}$,
\[\big(\omega^{b(\mathfrak{b})} ([\ell_{x_1}]),\dots, \omega^{b(\mathfrak{b})} ([\ell_{x_n }])\big)\overset{(d)}= (\omega_{x_1},\dots, \omega_{x_n}).
\]
Let $b'= b^{\epsilon_1}_{i_1}\dots b^{\epsilon_{k-1}}_{i_{k-1}}$ and $\mathfrak{b}'=b^{\epsilon_k}_{i_k}(\mathfrak{b})=([\ell'_x])_{x\in \mathcal{P}}$ , so that $b(\mathfrak{b})=b'(\mathfrak{b}')$.

We assume $\epsilon_k=1$ for simplicity, the case $\epsilon_k=-1$ being deal with similarly. The elements of $\mathfrak{b}'$ are $[\ell'_x]=[\ell_x]$ for $x\neq x_{i_k}$ and $[\ell'_x]=[\ell_{x_{i_{k+1}}}\ell_{x_{i_{k}}}\ell_{x_{i_{k+1}}}^{-1}]$, so that
\begin{align*}
([\ell_{x_1}],\dots, [\ell_{x_n}])&=([\ell'_{x_1}],\dots , [\ell'_{x_{i_k-1}}],[{\ell'}^{-1}_{x_{i_{k+1}}}\ell'_{x_{i_{k}}}\ell'_{x_{i_{k+1}}}],   [\ell'_{x_{i_k+1}}],\dots, [\ell'_{x_n}]),\\
\shortintertext{From the recursion assumption, we know that}
\big(\omega^{b(\mathfrak{b} )} ([\ell'_{x_1}]),\dots, \omega^{b(\mathfrak{b} )} ([\ell'_{x_n }])\big)&=
\big(\omega^{b'(\mathfrak{b}' )} ([\ell'_{x_1}]),\dots, \omega^{b'(\mathfrak{b}')} ([\ell'_{x_n }])\big)\overset{(d)}=(\omega_{x_1},\dots, \omega_{x_n}),\\
\shortintertext{hence}
\big(\omega^{b(\mathfrak{b} )} ([\ell_{x_1}]),\dots, \omega^{b(\mathfrak{b} )} ([\ell_{x_n }])\big)&\overset{(d)}=( \omega_{x_1}, \dots, \omega_{x_{i_k-1}}, \omega_{x_{i_k+1}}^{-1}\omega_{x_{i_k}}\omega_{x_{i_k+1}}, \omega_{x_{i_k+1}}, \dots, \omega_{x_n}).\\
\shortintertext{Thus, we are left to show that}
(\omega_{x_1},\dots, \omega_{x_n}) &\overset{(d)}=(\omega_{x_1},\dots, \omega_{x_{i-1}}, \omega_{x_{i+1}}^{-1}\omega_{x_{i}} \omega_{x_{i+1}}, \omega_{x_{i+1}},\dots, \omega_{x_n}).\\
\shortintertext{From independence, it suffices to show that}
(\omega_{x_i},\omega_{x_{i+1}})&\overset{(d)}= (\omega_{x_{i+1}}^{-1}\omega_{x_{i}} \omega_{x_{i+1}}, \omega_{x_{i+1}} ).
\end{align*}
For $j\in \{i,i+1\}$, let $\mathbb{P}^j$ be the law of $\omega_{x_j}$. Then, for all bounded $\mathbb{P}^i\otimes \mathbb{P}^{i+1}$-measurable function $f$ from $G^2$ to $\R$,
\begin{align*}
\mathbb{E}[f(\omega_{x_i},\omega_{x_{i+1}})]&=
\int_G \Big( \int_G f(u,v) \d \mathbb{P}^{i}_u \Big) \d \mathbb{P}^{i+1}_v\\
&= \int_G \Big( \int_G f(v^{-1}uv,v) \d \mathbb{P}^{i}_u \Big) \d \mathbb{P}^{i+1}_v \qquad \mbox{(from conjugation invariance of $\mathbb{P}^i$)}\\
&=
\mathbb{E}[f(\omega_{x_{i+1}}^{-1}\omega_{x_{i}}\omega_{x_{i+1}}, \omega_{x_{i+1}})].
\end{align*}
Thus  $(\omega_{x_i},\omega_{x_{i+1}})\overset{(d)}= (\omega_{x_{i+1}}^{-1}\omega_{x_{i}} \omega_{x_{i+1}}, \omega_{x_{i+1}} )$, which concludes the recursion and hence the proof.


%
%
%
%
\end{proof}
\begin{remark}
We carefully avoid to write $ \big(\omega^{b(\mathfrak{b})} ([\ell_{x_1}]),\dots, \omega^{b(\mathfrak{b})} ([\ell_{x_n }])\big)$ as a function of $b$ and $\big(\omega^{\mathfrak{b}} ([\ell_{x_1}]),\dots, \omega^{\mathfrak{b}} ([\ell_{x_n }])\big)$.
It seems that no simple expression for it can be given (or at least we failed to find it), because of the lack of commutativity between the actions of the Braid group and the symmetric group.
\end{remark}
This lemma will allow us to chose a specific proper basis $\mathfrak{b}$ without breaking the gauge invariance or the invariance by volume-preserving diffeomorphisms, provided that this choice of basis is independent from $\mathcal{P}_K^G$ conditional to $\mathcal{P}_K$.
Though, we first need to extend our definitions to the case when $\mathcal{P}$ is no longer finite but only locally finite.

\subsection{Locally finite set of punctures}
\begin{definition}
\label{def:proper}
Let $\mathcal{P}$ be a locally finite subset of $\R^2$. For all $x\in \mathcal{P}$, let $\gamma_x$ be a path from $0$ to $x$, with no-self-intersection, and such that for all $x,y\in \mathcal{P}, x\neq y$, $\gamma_x$ and $\gamma_y$ only intersects at $0$. Let $\ell_x$ be the loop defined as in the beginning of the section.

We say that the family $([\ell_x])_{x\in \mathcal{P}}$ is \emph{proper} if it generates $\pi_1(\mathcal{P})$. We set $\mathfrak{B}$ the set of proper basis.
\end{definition}
\begin{remark}
The last condition is \textbf{not} superfluous. The subtlety is that the group generated by $([\ell_x])_{x\in \mathcal{P}}$ is a direct limit of groups, whilst $\pi_1(\mathcal{P})$ is an inverse limit, in general larger. As a counter example, one can consider the case pictured in Figure \ref{figure:counter} below.

Besides, neither the direct limit of the finite braid groups, nor $\pi_0(\Aut(\R^2\setminus\mathcal{P}))$ acts transitively on $\mathfrak{B}$ anymore. Figure \ref{fig:counter2} shows two proper basis not related by any such braids.
\end{remark}
\begin{figure}[h]
\includegraphics{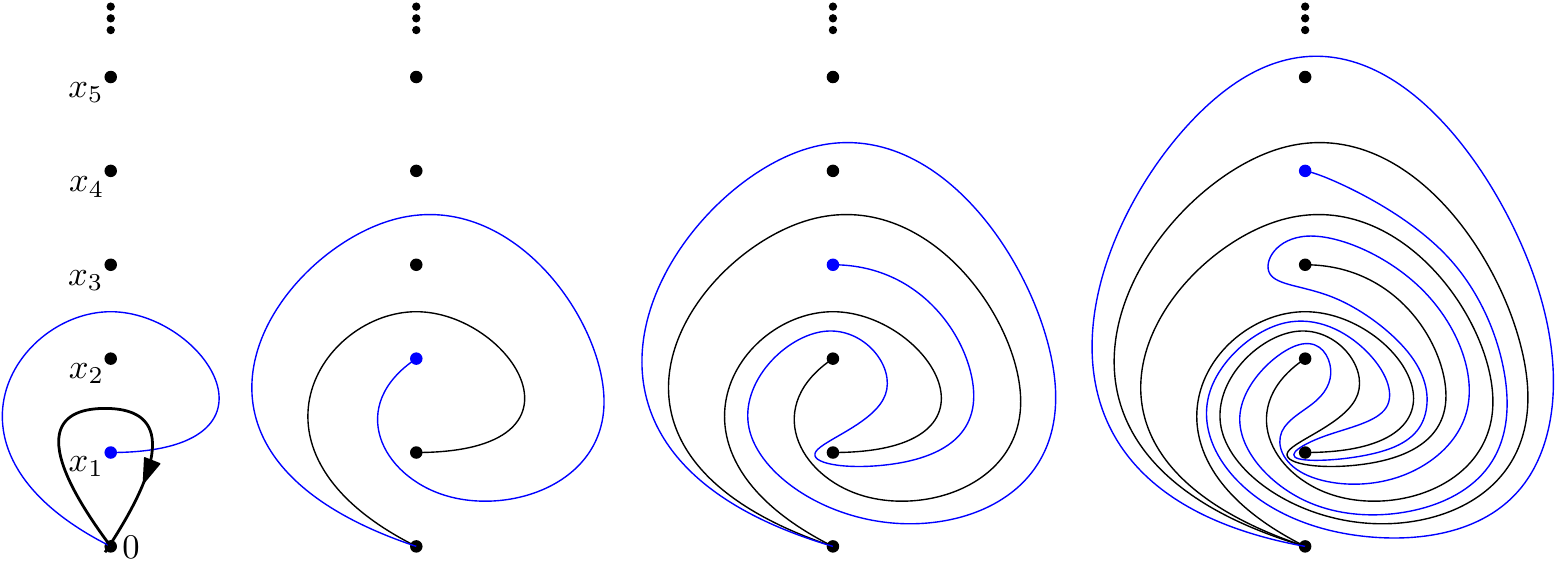}
\caption{\label{figure:counter} When $\mathcal{P}$ is no longer finite, the family $([\ell_x])_{x\in \mathcal{P}}$ can be free but not generating. In this example, the paths $\gamma_{x_i}$ are drawn one after the other. The path going to $x_{i+1}$ passes above the point $x_{i+2}$, and then navigates until it reaches $x_{i+1}$. The black loop which appears on the first picture does not lie on the group generated algebraically by the $[\ell_{x_i}]$. Formally, it would be describe as $w [\ell_{x_1}] w^{-1}$ with $w$ a bi-infinitely long word in the $[\ell_{x_i}]$.}
\includegraphics{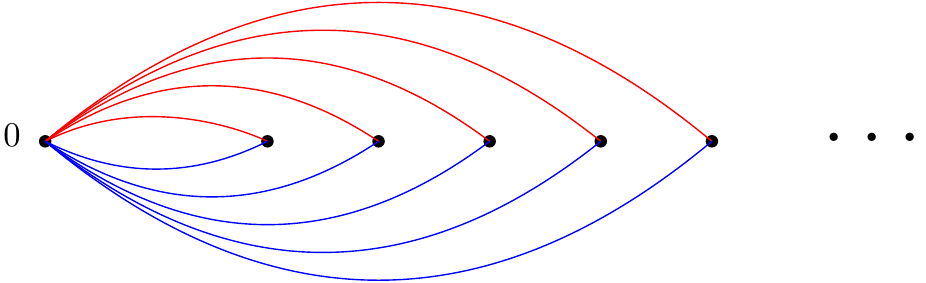}
\caption{\label{fig:counter2} The two proper basis pictured on red and blue are not related by the action of any braid. If such a braid were to exists, it would have to map the $n^{\mbox{\scriptsize{th}}}$ points to the  `$(\infty-n)^{\mbox{\scriptsize{th}}}$' one.}
\end{figure}
A proper basis  $\mathfrak{b}=([\ell_x])_{x\in \mathcal{P}}$ induces an isomorphism between $\pi_1(\mathcal{P})$ and the free group $\mathbb{F}_{ \mathcal{P} }$, and therefore allows to extend the definition of $\omega^{\mathfrak{b}}$ to the case when $\mathcal{P}$ is locally finite.

\begin{lemma}
\label{le:invariance:inf}
Lemma \ref{le:invariance} can be extended to locally finite subsets $\mathcal{P}$.
\end{lemma}
\begin{proof}
During this proof, for a subset $\tilde{\mathcal{P}}\subseteq \mathcal{P}$, we write $[\ell_x]_{\tilde{\mathcal{P}}}$ the homotopy class of $\ell_x$ on $\R^2\setminus \tilde{\mathcal{P}}$.

Let $\mathfrak{b}=([\ell_x])_{x\in \mathcal{P}}$ and $\mathfrak{b}'$ be two proper basis.
To show that $(\omega^{\mathfrak{b}}([\ell]))_{\ell \in \pi_1(\mathcal{P})}\overset{(d)}=(\omega^{\mathfrak{b}'} ([\ell]) )_{\ell \in \pi_1(\mathcal{P})}$,
it suffices to show that for any \emph{finite} subset $\mathcal{P}'\subseteq \mathcal{P}$, \[(\omega^{\mathfrak{b}}([\ell_x]))_{x\in \mathcal{P}'}\overset{(d)}=(\omega^{\mathfrak{b}'}([\ell_x]))_{x\in \mathcal{P}'}.\]
Remark the left-hand side is equal to $(\omega_x)_{x\in \mathcal{P}'}$.

Since $\mathfrak{b}'$ is a generating family for $\pi_1(\mathcal{P})$, for all $x\in \mathcal{P}'$, there exists $n_x\in \mathbb{N}$, $y_{x,1},\dots, y_{x,n_x}\in \mathcal{P}$ and $\epsilon_{x,1},\dots, \epsilon_{x,n_x}\in \{\pm 1\}$ such that
\[ [\ell_x]_{\mathcal{P}}=[\ell'_{y_{x,1}}]_{\mathcal{P}}^{\epsilon_{x,1}}\dots [\ell'_{y_{x,n_x}}]_{\mathcal{P}}^{\epsilon_{x,n_x}}.\]
Then, \[ \omega^{\mathfrak{b}'}([\ell_x])=\omega_{y_{x,1}}^{\epsilon_{x,1}}\dots \omega_{y_{x,n_x}}^{\epsilon_{x,n_x}}.\]
Let $\mathcal{P}''$ be the finite set $\{ y_{x,i}, x\in \mathcal{P}', i\in \{1,\dots, n_x\}  $.
The homotopy class of $\ell_x$ in $\pi_1(\mathcal{P}'')$ is given by
\[ [\ell_x]_{\mathcal{P}''}=\pi( [\ell_x]_{\mathcal{P}})=\pi([\ell'_{y_{x,1}}]_{\mathcal{P}}^{\epsilon_{x,1}}\dots [\ell'_{y_{x,n_x}}]_{\mathcal{P}}^{\epsilon_{x,n_x}} )  =[\ell'_{y_{x,1}}]_{\mathcal{P}''}^{\epsilon_{x,1}}\dots [\ell'_{y_{x,n_x}}]_{\mathcal{P}''}^{\epsilon_{x,n_x}},\]
where $\pi$ is the canonical projection $[\ell]_{\mathcal{P}}\mapsto[\ell]_{\mathcal{P}''}$ from $\pi_1(\mathcal{P})$ to $\pi_1(\mathcal{P}'')$.

Let $\mathfrak{b}_{|\mathcal{P}''}=([\ell_x]_{\mathcal{P}''})_{x\in\mathcal{P}'' }$ and $\mathfrak{b'}_{|\mathcal{P}''}=([\ell'_x]_{\mathcal{P}''})_{x\in\mathcal{P}'' }$, which are two proper basis for the set $\mathcal{P}''$.
Since $\mathcal{P}''$ is finite, we can apply Lemma \ref{le:invariance}, which states that $\big(\omega^{\mathfrak{b}_{|\mathcal{P}''}}
([\ell_x]_{\pi_1(\mathcal{P}'')    })
  \big)_{ x\in \mathcal{P}'}\overset{(d)}= \big(\omega^{\mathfrak{b'}_{|\mathcal{P}''}}([\ell_x]_{\pi_1(\mathcal{P}'')    })  \big)_{ x\in \mathcal{P}'} $,
but these are nothing else than
$(\omega_x)_{x\in \mathcal{P}'}$ and $(\omega^{\mathfrak{b}'}([\ell_x]) )_{ x\in \mathcal{P}'}$.
\end{proof}

\begin{definition}
Assume that $\R^2$ is endowed with an Euclidean structure, and that no two points of $\mathcal{P}$ are colinear with $0$.
Then, a proper basis $\mathfrak{b}_0$ is obtained by taking, for all $x\in \mathcal{P}$, as the path $\gamma_x$, the straight line segment from $0$ to $x$.


We then set $\omega_K=\omega^{\mathfrak{b}_0}_K$.
\end{definition}
In the following sections, we will identify $\pi_1(\mathcal{P})$ with $\mathbb{F}_{\mathcal{P}}$ using this proper basis $\mathfrak{b}_0$.

\begin{remark}
For any locally finite set of point $\mathcal{P}$, there exists a Euclidean structure such that this condition is fullfilled, so that $\mathfrak{B}$ is always non-empty.

For a given Euclidean structure and for $\mathcal{P}=\mathcal{P}_K$ a Poisson random set whose intensity admits a density with respect to the Lebesgue measure, the condition that no two points are colinear is fullfilled with probability $1$. In the following, we always assume this full probability event to hold.
\end{remark}

In the following proposition, we no longer consider the set $\mathcal{P}$ nor the family $(\omega_x)_{x\in \mathcal{P}}$ to be fixed. For $\mathcal{P}^G=\{(x,\omega_x)\}_{x\in \mathcal{P}}$, we write $\omega(\mathfrak{b},\mathcal{P}^G)$ for the map previously written $\omega^{\mathfrak{b}}$.
\begin{proposition}
The distribution of $\omega(\mathfrak{b},\mathcal{P}^G_K)$ does not depend on the $\sigma(\mathcal{P}_K)$-measurable proper basis $\mathfrak{b}$.

Besides, it is invariant by volume-preserving diffeomorphisms.
To be more specific, for all deterministic volume-preserving diffeomorphism $\phi$ of $\R^2$ such that $\phi(0)=0$, and for all continuous functions $\ell_1,\dots, \ell_k:[0,1]\to \R^2$ with $\ell_i(0)=\ell_i(1)=0$ for $i\in \{1,\dots, k\}$, and whose ranges has vanishing Lebesgue measure,
\[
\big(\omega(\mathfrak{b},\mathcal{P}^G_K)[\ell_1], \dots,\omega(\mathfrak{b},\mathcal{P}^G_K)[\ell_k]\big)\overset{(d)}=
\big(\omega(\mathfrak{b},\mathcal{P}^G_K)[\phi\circ \ell_1] , \dots, \omega(\mathfrak{b},\mathcal{P}^G_K)[\phi\circ \ell_k] \big).
\]
\end{proposition}
Remark that the vanishing range assumption is necessary for $\omega(\mathfrak{b},\mathcal{P}^G_K)[\ell_i]$ to be almost surely well-defined.
\begin{proof}
Let $\mathfrak{b}, \mathfrak{b}'$ be two  $\sigma(\mathcal{P}_K)$-measurable proper basis. For $x\in \mathcal{P}_K$, let $\omega_x\in G$ be the unique element such that $(x,\omega_x)\in \mathcal{P}_K^G$. Then, conditionally on $\sigma(\mathcal{P}_K)$, the variables $(\omega_x)_{x\in \mathcal{P}_K}$ are independent and each have a law invariant by conjugation. We can therefore apply Lemma \ref{le:invariance:inf}. We concluded that conditionally on $\sigma(\mathcal{P}_K)$,
$\omega(\mathfrak{b},\mathcal{P}^G_K)$ and $\omega(\mathfrak{b}',\mathcal{P}^G_K)$ have the same distributions. Integrating on $\sigma(\mathcal{P}_K)$, we deduce that
$\omega(\mathfrak{b},\mathcal{P}^G_K)$ and $\omega(\mathfrak{b}',\mathcal{P}^G_K)$ have the same distributions.

 Let now $\phi$ be a volume-preserving diffeomorphism. Then, $\tilde{\phi}=\phi\times \id_G$ preserves the measure $K\lambda \otimes \mu_K$ on $\R^2\times G$, so that the sets $\mathcal{P}_K^G$ and $\tilde{\phi}(\mathcal{P}_K^G)$ have the same distribution. Since the distribution of $\omega(\tilde{\mathfrak{b}}, \phi(\mathcal{P}^G_K))$ does not depend on the $\sigma(\mathcal{P}_K)$-measurable
 proper basis $\tilde{\mathfrak{b}}$ but only on $\phi(\mathcal{P}^G_K)$, and since $\phi(\mathfrak{b})$ is a $\sigma(\mathcal{P}_K)$-measurable
 proper basis,
 \[\omega(\phi(\mathfrak{b}), \tilde{\phi}(\mathcal{P}^G_K))\overset{(d)}= \omega( \tilde{\phi}(\mathcal{P}^G_K))\overset{(d)}=\omega( \mathcal{P}^G_K)\overset{(d)}=\omega(\mathfrak{b},\mathcal{P}^G_K )\]
(the two middle objects being only defined in distribution).
%
%

We conclude by noticing that  \[(\omega(\mathfrak{b}, \mathcal{P}_K^G )[\phi\circ \ell])_{\ell}\overset{(d)}=  ( \omega(\phi(\mathfrak{b}), \tilde{\phi}(\mathcal{P}_K^G) )[\phi\circ \ell])_{\ell}=(\omega(\mathfrak{b},\mathcal{P}_K^G)[\ell])_{\ell}.\]
\end{proof}

\section{A much simpler problem}
\label{sec:simpler}
%

Let us recall that our goal is to study $\omega_K([\bar{X}])$, which can be written as a very large product of random elements in $G$,
\[  \omega_K([\bar{X}])=g_{i_1}\dots g_{i_k}.
\]
If the group $G$ was commutative, we could arrange this product as
\begin{equation} g_1^{\theta(x_1,\bar{X})}\dots g_n^{\theta(x_n,\bar{X})},\label{eq:word} \end{equation}
where $(x_i,g_i)$ is an enumeration of our random set $\mathcal{P}^G_K$ and $\theta(x,\bar{X})$ is the integer winding number of $\bar{X}$ around $x$. Our goal, as we explained sooner, is actually to reduce our problem to the study of this simpler expression \eqref{eq:word}, even when $G$ is not commutative. In this section, we show that the expression given by \eqref{eq:word} does converge in distribution, as $K\to \infty$, and we identify the limit.
\begin{definition}
We say that a probability measure $\nu$ on $\mathfrak{g}$ lies in the \emph{strong attraction domain} of  $\nu^\sigma$, if $\nu$ is $\ad$-invariant and there exists $\delta>0$ and a coupling  $(X,Y): X\sim\nu,\ Y\sim\nu^\sigma$ such that $\mathbb{E}[|X-Y|^{1+\delta}]<+\infty$ and $\mathbb{E}[X-Y]=0$.
We then write $\nu\in \mathcal{D}_\delta(\nu^\sigma)$, and we write $[\nu,\nu^{\sigma}]$ for the law of such a coupling $(X,Y)$.
\end{definition}

\begin{proposition}
  \label{prop:orderedCauchy}
  Let $G$ be a compact Lie group with Lie algebra $\mathfrak{g}$ endowed with a biinvariant scalar product. Let $\nu\in \mathcal{D}_\delta(\nu^\sigma)$, for some $\sigma>0$.

  Let $(X_i)$ be an i.i.d. sequence of random variables distributed according to $\nu$. 
  Then, the random variable \[\exp(\tfrac{X_1}{n})\dots \exp(\tfrac{X_n}{n})\]
  converges in distribution, as $n\to +\infty$. The limit distribution only depends on $\sigma$.
\end{proposition}
Let us give a detailed version of this proposition, with intermediary results that will guide us through the proof.
For a $\mathcal{C}^1$ curve $\Gamma:[0,1]\to T_1G\simeq \mathfrak{g}$ starting at $0$, consider the ordinary differential equation
\[
\gamma(0)=1 \qquad \frac{\d }{\d t}(\gamma(t_0)^{-1} \gamma(t) )_{|t=t_0}=\Gamma'(t).
\]
We know from Young's integration theory that the solution map $I:\Gamma\mapsto \gamma$ extends continuously in the $p$-variation norms for all $p<2$. With still write $I$ for the extended map.

\begin{proposition*}
  Let $G$ be a compact Lie group with Lie algebra $\mathfrak{g}$ endowed with a biinvariant scalar product. Let $\nu\in \mathcal{D}_\delta(\nu^\sigma)$, for some $\sigma>0$.

  Set $(X_i,X^{\sigma}_i)$ an i.i.d. sequence distributed according to $[\nu,\nu^\sigma]$. For $n\in \mathbb{N}\setminus\{0\}$, set $\Gamma^{(n)}$ and $\Gamma^{(n,\sigma)}$ the continuous functions from $[0,1]$ to $\mathfrak{g}$ given by
  \begin{align*}
  \Gamma^{(n)}(t)&=\frac{1}{n} \sum_{i=1}^{\lfloor t n\rfloor} X_i+(tn-\lfloor tn\rfloor) X_{\lfloor tn\rfloor+1},\\
  \Gamma^{(n,\sigma)}(t)&=\frac{1}{n} \sum_{i=1}^{\lfloor t n\rfloor} X^{\sigma}_i+(tn-\lfloor tn\rfloor) X^{\sigma}_{\lfloor tn\rfloor+1}.
  \end{align*}
  \begin{itemize}
  \item  For all $p>1$, $\Gamma^{(n,\sigma)}$ converges weakly in $p$-variation. Let $\Gamma^{(\sigma)}$ be distributed as the limit.
  \item For all $p\in (1,1+\delta)$, $ \|d_{p-var}(\Gamma^{(n)}, \Gamma^{(n,\sigma)} )\|_{L^1(\Omega)} \underset{n\to\infty}\to 0$. In particular, $\Gamma^{(n)}$ converges weakly in $p$-variation toward
  $\Gamma^{(\sigma)}$.
  \item 
  For all $n$ and all
  $i\in \{0,\dots,n \}$,
  $I(\Gamma^{(n)})(\tfrac{i}{n})= \exp_G(\tfrac{X_1}{n} )\dots \exp_G(\tfrac{X_i}{n})$.
  \item 
  As $n\to \infty$, the random variable $\exp_G(\tfrac{X_1}{n} )\dots \exp_G(\tfrac{X_i}{n})$ converges in distribution, and the limit law is the one of $I(\Gamma^{(\sigma)})(1)$.
  \end{itemize}
\end{proposition*}
\begin{proof}
   Let $\Gamma^{\sigma}$ be a symmetric $1$-stable process such that $\Gamma^{\sigma}(1)$ is distributed as $\nu^\sigma$.


  Then, for any $n\in \mathbb{N}\setminus\{0\}$, the family $( \frac{1}{n} \sum_{i=1}^{j} X^{\sigma}_i )_{j\in\{0,\dots, n\} }$ is distributed as $(\Gamma^{\sigma}(\tfrac{j}{n}))_{j\in\{0,\dots, n\} }$.
  The piecewise linear curve with interpolation points the $(\Gamma^{\sigma}(\tfrac{j}{n}))_{j\in\{0,\dots, n\} }$ is therefore distributed as $\Gamma^{(n,\sigma)}$. Thus, in some probability space, we can assume that each curve $\Gamma^{(n,\sigma)}$ is exactly the piecewise linear approximation of $\Gamma^{\sigma}$ with the interpolation times $\{0,\tfrac{1}{n},\dots , 1\}$.

  Since $\Gamma^{\sigma}$ has finite $q$-variation for all $q>1$ \cite[Theorem 4.1]{Blumenthal},
  it lies on the $p$-variation completion of $\mathcal{C}^\infty$ for all $p>1$ \cite[Corollary 5.35]{frizVictoir}. By the Wiener's characterization  \cite[Theorem 5.33, $\it{(i.1)\implies (i.3)}$]{frizVictoir}),
  for all $q>1$, $d_{q-var}( \Gamma^{(n,\sigma)}, \Gamma^{\sigma})\underset{n\to \infty}\longrightarrow 0$. This proves the first item, and shows besides that $\Gamma^{(\sigma)}\overset{(d)}=\Gamma^{\sigma}$ is a stable process associated with $\nu^\sigma$.

  For the second item, we use the following interpolation inequality \cite[Proposition 5.5]{frizVictoir}: for all $p\in (1,+\infty)$, for all $f$,
  \begin{equation}
  \label{eq:interpolation}
   \|f\|_{p-var}\leq \|f\|_{1-var}^{\frac{1}{p}}\|f\|_0^{1-\frac{1}{p}}
   \end{equation}
  where $\|f\|_0=\sup_{s,t\in [0,1]} |f(s)-f(t)|$.

  We also use the following inequality \cite[Lemma 2]{Rosenthal}\footnote{Despite the title of the cited article, the given inequality \emph{does} concern the case $p<2$.}. 

  \begin{lemma}
  Let $X_1,\dots, X_n$ be a family of independant and centered $\mathbb{R}$-valued random variables with moment of order $p<2$. Then,
  \begin{equation}
  \label{eq:rosenthal}
   \mathbb{E}\big[\big|\sum_{i=1}^n X_i\big|^p \big]^{\frac{1}{p}}\leq 2 \big(\sum_{i=1}^n \mathbb{E}[|X_i|^p ]\big)^{\frac{1}{p}}.
   \end{equation}
  \end{lemma}
  We then obtain, for $q\in(1,1+\delta)$ and $p>1$:
  \begin{align*}
  \mathbb{E}[ \|\Gamma^{n} -\Gamma^{(n,\sigma)}\|_{p-var} ]
  &\leq \mathbb{E}[ \|\Gamma^{n} -\Gamma^{(n,\sigma)}\|_{1-var}^{\frac{1}{p}} \|\Gamma^{n} -\Gamma^{(n,\sigma)}\|_{0}^{1-\frac{1}{p}} ] \qquad \mbox{ (using \eqref{eq:interpolation})}\\
  &\leq \mathbb{E}[ \|\Gamma^{n} -\Gamma^{(n,\sigma)}\|_{1-var}]^{\frac{1}{p}} \mathbb{E}[\|\Gamma^{n} -\Gamma^{(n,\sigma)}\|_{0}]^{1-\frac{1}{p}} \qquad \mbox{(H\"older's inequality)} \\
  &\leq \big(\sum_{i=1}^n \tfrac{1}{n}\mathbb{E}[|X_i-Y_i|] \big)^{\frac{1}{p}} \big(2  \mathbb{E}[\|\Gamma^{n} -\Gamma^{(n,\sigma)}\|_{\infty}] \big)^{1-\frac{1}{p}} \qquad \mbox{($\|f\|_0\leq 2\|f\|_\infty $)}\\
  &= \mathbb{E}[|X_1-Y_1|]^{\frac{1}{p}}  2^{1-\frac{1}{p}} \mathbb{E}[ \max_{i\in \{0,\dots, n\}} \big| \sum_{j=1}^i \tfrac{1}{n} (X_j-Y_j) \big| ]^{1-\frac{1}{p}}\\
  &\leq C_p
  \big( \tfrac{q}{q-1} \mathbb{E}\big[\big|\sum_{j=1}^n \tfrac{1}{n}(X_j-Y_j)\big|^q \big]^{\frac{1}{q}} \big)^{ 1-\frac{1}{p}}
  \ \mbox{(Doob's maximal inequality)}\\
  &\leq C_{p,q}
  \Big(2  \big(\sum_{j=1}^n \mathbb{E}\big[ \tfrac{1}{n^q}|X_j-Y_j|^q \big]\big)^{\frac{1}{q}}  \Big)^{ 1-\frac{1}{p}} \ \mbox{(using \eqref{eq:rosenthal})}\\
  &\leq  C'_{p,q} n^{-\frac{(p-1)(q-1)}{pq}  } \underset{n\to +\infty}\longrightarrow 0.
  \end{align*}
  This proves the second item.

  The third one is a very basic computation.  

  The last one follows from the continuity of $I$ in $p$-variation for $p\in(1,2)$.
\end{proof}
Before we explain how we will use Proposition \ref{prop:orderedCauchy}, we need the following definition
\begin{definition}
For a point $x$ outside the range of $\bar{X}$, the planar Brownian motion concatenated with a straight segment, we define $\theta(x)\in \mathbb{Z}$ the winding number of $\bar{X}$ around the point $x$.
\end{definition}
\begin{theorem}
\label{th:CauchyDomTheta}
  Let $\mathbb{P}^X$ be the law of the Brownian motion (from $[0,1]$ to $\R^2$).
  For $R>0$, let $x$ be a point distributed uniformly on $B(0,R)$. Then, $\mathbb{P}^X$-almost surely, the random variable $\theta(x)$ lies on the strong attraction domain of a Cauchy distribution with scale parameter $\frac{1}{2}$.
\end{theorem}
\begin{proof}
This is a direct corollary of Theorem 1.1 in our previous paper \cite{LAWA}.
\end{proof}
\begin{lemma}
\label{le:rot}
  Let $R$ be a random variable on the strong attraction domain $\mathcal{D}_\delta(C)$ of a (real-valued) Cauchy distribution $C$ with scale parameter $s$.
  Let $U$ be a random variable independent from $X$, and distributed uniformly on the unit sphere of $\mathfrak{g}$. Then, $RU$ lies on the attraction domain
  $\mathcal{D}_\delta(\nu^\sigma)$ of the symmetric $1$-stable distribution $\nu^\sigma$ with $\sigma= \frac{ \Gamma\big( \frac{d}{2}\big)}{\sqrt{\pi} \Gamma \big(\frac{d+1}{2} \big)}s$.
\end{lemma}
\begin{proof}
  One can and will assume that $s=1$. The general cases is then recovered from a scaling argument.

  Let $S$ be a Cauchy random variable which is such that $\mathbb{E}[|R-S|^{1+\delta}]<+\infty$. Since $\mathbb{E}[\|RU-SU\|^{1+\delta}]=\mathbb{E}[|R-S|^{1+\delta}]<+\infty$, it suffices to show that there exists a random variable $Z$ distributed as $\nu^\sigma $ with $\mathbb{E}[\|SU-Z\|^{1+\delta}]<+\infty$. Actually, we will show that there exists $Z$
  distributed as $\nu^\sigma$ and such that $SU-Z$ is bounded.
%

  For all $x\in [0,+\infty)$, let $\psi(x)$ be the unique positive real number solution of the equation
  \begin{equation}
  \label{eq:phi}
\frac{2}{\pi}  \int_{\psi(x)}^{+\infty} \frac{\d z}{1+z^2}=\frac{2 \Gamma\big(\tfrac{d+1}{2}\big)  }{\sqrt{\pi} \Gamma\big(\tfrac{d}{2}\big)}
  \int_{\sigma^{-1} x}^{+\infty} \frac{z^{d-1} \d z}{(1+z^2)^{\frac{d+1}{2} }}.\end{equation}
  The constants are tuned so that both sides are equal to $1$ for $x=0=\psi(0)$.
  It is then easily seen that $\psi$ defines a continuous increasing bijection of $[0,+\infty)$. Set $\phi=\psi^{-1}$. We set $Z_\phi=\sgn(S)\phi(|S|)U$. We also set $Z_\sigma$ a random variable distributed as $\nu^\sigma$.

Remark that $\|Z_\sigma\|$ and $\frac{Z_\sigma}{\|Z_\sigma\|})$ are independent variables, and that the latter is distributed uniformly over the unit sphere. The same is true for $Z_\sigma$ replaced with $Z_\phi$, and to show that $Z_\sigma\overset{(d)}= Z_\phi$ thus reduces to show that  $\|Z_\sigma\|\overset{(d)}= \|Z_\phi\|$, which follows from a simple computation. Indeed,
\[
\mathbb{P}( \|Z_{\phi}U\|\geq r )= 2 \mathbb{P}( Z_{\phi} \geq r  )
= 2\mathbb{P}(S\geq \psi(r))
= \frac{2}{\pi}  \int_{\psi(r)}^{+\infty } \frac{\d z}{1+z^2}
=2\frac{\Gamma\big(\tfrac{d+1}{2}\big)  }{\sqrt{\pi} \Gamma\big(\tfrac{d}{2}\big)}
  \int_{\sigma^{-1}r}^{+\infty} \frac{z^{d-1} \d z}{(1+z^2)^{\frac{d+1}{2} }}.
\]
On the other hand,
\begin{align*}
\mathbb{P}( \|Z_{\sigma} \|\geq r )&= \int_{r}^{+\infty} |S^{d-1}| \frac{z^{d-1} \Gamma\big(\tfrac{d+1}{2}\big) \d z}{\pi^{\frac{d+1}{2}}\sigma^d (1+\sigma^{-2} z^2)^{\tfrac{d+1}{2} }  }
=\frac{2\pi^{\frac{d}{2}}\Gamma\big(\tfrac{d+1}{2}\big) }{\pi^{\frac{d+1}{2}} \Gamma\big(\tfrac{d}{2}\big)}
\int_{\sigma^{-1}r}^{+\infty} \frac{u^{d-1}\d u}{(1+u^2)^{\frac{d+1}{2}}  }\\
&=\mathbb{P}( \|Z_{\phi}U\|\geq r ).
\end{align*}
It follows that $Z_\phi$ is indeed distributed as $Z_\sigma$.

Let us now tune the parameter $\sigma$ so that $Z_\phi=\sgn(S)\phi(|S|)$ and $S$ are close to each other.

  Since \[\int_{\psi(x)}^{+\infty} \frac{\d z}{1+z^2}\sim \frac{1}{\psi(x)}
  \qquad \mbox{ and }\qquad
  \int_{\sigma^{-1}x}^{+\infty} \frac{z^{d-1} \d z}{(1+z^2)^{\frac{d+1}{2} }}\sim \frac{1}{\sigma^{-1}x},
  \]
  Equation \eqref{eq:phi} implies that
  $\frac{2}{\pi \psi(x)}\sim \tfrac{2 \Gamma \big(\tfrac{d+1}{2} \big)}{\sqrt{\pi} \Gamma\big( \frac{d}{2}\big) \sigma^{-1} x}$, or equivalently that $\psi(x)\sim \tfrac{ \Gamma\big( \frac{d}{2}\big) x}{\sqrt{\pi} \Gamma \big(\tfrac{d+1}{2} \big)} \sigma^{-1} x$, or equivalently that
  $\tfrac{\phi(x)}{x}\underset{x\to \infty}\longrightarrow \tfrac{\sqrt{\pi} \Gamma \big(\tfrac{d+1}{2} \big)}{ \Gamma\big( \frac{d}{2}\big)} \sigma$.

For $Z_\phi$ and $S$ to be close to each other, including when $S$ is large, this limit must be one, so that we must set $\sigma\coloneqq \tfrac{ \Gamma\big( \frac{d}{2}\big)}{\sqrt{\pi} \Gamma \big(\tfrac{d+1}{2} \big)}$.

A simple computation gives
\[ \psi(x)\int_{\psi(x)}^{\infty} \frac{\d z}{1+z^2}\underset{x\to +\infty}=1+O(x^{-2}), \quad \sigma^{-1}x\int_{\sigma^{-1} x}^{+\infty} \frac{z^{d-1}}{(1+z^2)^{\frac{d+1}{2} } }\d z\underset{x\to +\infty}= 1+O(x^{-2} ).
\]
We deduce that
\begin{align*}
\psi(x)
&=\Big(\int_{\psi(x)}^{\infty} \frac{\d z}{1+z^2}\Big)^{-1}(1+O(x^{-2}))
=\Big( \sigma^{-1} \int_{\sigma^{-1} x}^{+\infty} \frac{z^{d-1}}{(1+z^2)^{\frac{d+1}{2} } }\d z  \Big)^{-1}
(1+O(x^{-2}))\\
&= \Big(\tfrac{1}{x}(1+O(x^{-2})) \Big)^{-1}(1+O(x^{-2}))=x+O(x^{-1}).
\end{align*}
It follows that $\psi(x)-x$ is bounded near $+\infty$, and by symmetry it is bounded on $\R$. Thus, $Z_\phi-S$ is bounded, and so is $Z_\phi U-SU$.
\end{proof}

\begin{corollary}
\label{coro:4.7}
Let $\iota_K:\mathcal{P}_K\to \mathbb{N}$ be any bijection, independant from $X$ and $\mathcal{P}_K^G$ conditional to $\mathcal{P}_K$. Then,
$\mathbb{P}_X^R$-almost surely, as $K\to \infty$, the product \[\prod_{(g,x)\in \mathcal{P}_K^G\cap B_R} g^{\theta(x)},\] ordered according to $\iota_K$, converges in distribution towards $\nu^{\sigma}$, for $\sigma=\frac{\Gamma\big(\tfrac{d}{2}\big)}{2\sqrt{\pi} \Gamma\big(\tfrac{d+1}{2}\big)}$.
\end{corollary}
\begin{proof}
Theorem \ref{th:CauchyDomTheta} and Lemma \ref{le:rot} ensures that $g^{\theta(x)}$ lies in $D_\delta(\nu^{\sigma})$.  Proposition \ref{prop:orderedCauchy} allows to conclude.
\end{proof}

The problem that we now face, in order to prove the main theorem, is to replace this product with $\omega_K(\bar{X})$.

\section{Free groups}
\label{sec:free}

\subsection{Free groups as semi-direct product}
%

Let $\mathcal{P}=\{x_1,\dots,x_k\}$ be a totally ordered finite set.
We will build up a group isomorphism between the free group $\mathbb{F}_{\mathcal{P}}$ and the semi-direct product of free groups
$\mathbb{F}_{\mathcal{P}\setminus\{x_k\}} \ltimes \mathbb{F}_{\mathbb{F}_{ \mathcal{P}\setminus\{x_k\} } }$.

Let $\iota_\mathcal{Q}$ be the canonical injection of a set $\mathcal{Q}$ into the free group $\mathbb{F}_{\mathcal{Q}}$. We always make implicit the inclusion $\iota_\mathcal{P}$, but we write explicitly $\iota=\iota_{\mathbb{F}_{\mathcal{P}\setminus\{x_k\} }}$.
\begin{lemma}
Let $\pi^{k}: \mathbb{F}_{ \mathcal{P}} \to \mathbb{F}_{ \mathcal{P}\setminus\{x_k\}}$ be the canonical projection, $H$ the smallest normal subgroup of $\mathbb{F}_{ \mathcal{P}}$ containing $x_k$, and $W=\iota(\mathbb{F}_{ \mathcal{P}\setminus\{x_k\} })$, that is the set of the generators of $ \mathbb{F}_{ \mathbb{F}_{ \mathcal{P}\setminus\{x_k\} } }$.
\begin{enumerate}
\item The kernel of $\pi^{k}$ is $H$.
\item The homomorphism $c_k: \mathbb{F}_{ \mathbb{F}_{ \mathcal{P}\setminus\{x_k\} } } \to H$ which maps $\iota(h)\in W$ to $hx_k h^{-1}$
is an isomorphism.
\end{enumerate}
Therefore,
\[
0\longrightarrow  \mathbb{F}_{ \mathbb{F}_{ \mathcal{P}\setminus\{x_k\} } } \overset{c_k}{\longrightarrow} \mathbb{F}_{ \mathcal{P}} \overset{\pi^{k}}
{\rightleftarrows  }
 \mathbb{F}_{ \mathcal{P}\setminus\{{x_k}\}}\longrightarrow 0,
\]
is a split short exact sequence and $\mathbb{F}_{ \mathcal{P}}\cong \mathbb{F}_{ \mathcal{P}\setminus\{{x_k}\}} \ltimes \mathbb{F}_{ \mathbb{F}_{ \mathcal{P}\setminus\{{x_k}\} } }$.
%
%
\end{lemma}
We have not found this result on the existing litterature, but the kernel of $A\star B\to A \times B$ is described in \cite{serre} in a similar fashion. This description probably implies our lemma, but we have not been able to derive one from the other in a simple way.
\begin{proof}
\emph{First point:}

Since $\ker(\pi^{k})$ is a normal subgroup and ${x_k}\in \ker(\pi^{k})$, $H\subseteq \ker(\pi^{k})$.


Let $g\in\mathbb{F}_{ \mathcal{P}}$. Then, there exists $\epsilon_1,\dots, \epsilon_n\in \{\pm 1\}$ and $h'_1,\dots, h'_{n+1}\in \mathbb{F}_{ \mathcal{P}\setminus \{{x_k}\} }$ such that
\[g=h'_1{x_k}^{\epsilon_1}h'_2\dots h'_n {x_k}^{\epsilon_n} h'_{n+1}.\]
Setting $h_i= h'_1\dots h'_i$, we get
\[g=(h_1{x_k}^{\epsilon_1}h_1^{-1}) \dots (h_n {x_k}^{\epsilon_n} h_n^{-1}) h_{n+1}.\]
Then, $\pi^{k}(g)=\pi^{k}(h_{n+1})=h_{n+1}$.  Therefore, $g\in \ker (\pi^{k})$ implies $h_{n+1}=1$, which implies that
$g\in \langle hx_kh^{-1}: h\in  \mathbb{F}_{ \mathcal{P}\setminus \{x_k\} } \rangle_{\mathbb{F}_{ \mathcal{P} }}=H$.

\emph{Second point:}

 We only need to check that $x_k$ is injective. Let us assume otherwise, in which case there exists a positive integer $n$, $h_1,\dots, h_n\in \mathbb{F}_{ \mathbb{F}_{ \mathcal{P} \setminus \{ x_k\} } }$ and $\epsilon_1,\dots ,\epsilon_n$ such that
\[
  h_1x_k^{\epsilon_1}h_1^{-1} \dots h_n x_k^{\epsilon_n} h_n^{-1}=1\in \mathbb{F}_{ \mathcal{P}}, \ h_1^{\epsilon_1}\dots h_n^{\epsilon_n}\neq 1\in \mathbb{F}_{ (\mathcal{P}\setminus \{x_k \}) }.
\]
We assume that $n$ is minimal for this property, and we will come to a contradiction by finding an element with the same property but with a smaller $n$.
We will look at
\[
  g =x_k^{-\epsilon_1} h_1^{-1}     = h_1^{-1}h_2x_k^{\epsilon_2}h_2^{-1} \dots h_n x_k^{\epsilon_n} h_n^{-1} .
\]
For each $h_i$, we fix a sequence $h_{i,1},\dots, h_{i,j_i}\in (\mathcal{P}\setminus\{x_k\} )^{\pm 1}$ such that $h_i=h_{i,1}\dots h_{i,j_i}$. This induces a path from $1$ to $h_i$ in the Cayley graph $\Gamma$ of $\mathbb{F}_{ \mathcal{P}} $, with respect to the generating family $\mathcal{P}$ (acting on the right). Concatenating these paths together, we obtain a path $g_1=1,\dots, g_n=g$
from $1$ to $g$ in $\Gamma$.
Since $\Gamma$ is a tree, any path on it that reaches $x_k^{-\epsilon_1} h_1^{-1}$ must actually reach $x_k^{-\epsilon_1} $. In particular, there exists $i$ such that
$g_i=x_k^{-\epsilon_1} $. If we assume $i$ to be minimal with this property, then necessarily $g_{i-1}^{-1}g_i= x_k^{-\epsilon_1}$ (the first time we reach
$x_k^{-\epsilon_1}$, it is necessarily by crossing the edge from $1$ to $x_k^{-\epsilon_1}$), so that
there exists $j$ such that $g_i= h_1^{-1}(h_2 x_k^{\epsilon_2} h_2^{-1})\dots h_j x_k^{\epsilon_j}$, with $\epsilon_j=-\epsilon_1$.

Therefore,
\begin{align}
1&= h_1 x_k^{\epsilon_1} g_i  h_j^{-1} h_{j+1} x_k^{\epsilon_{j+1}} h_{j+1}^{-1}\dots h_{n} x_k^{\epsilon_{n}} h_{n}^{-1}\nonumber\\
\label{eq:minim}&= h_1  h_j^{-1} h_{j+1} x_k^{\epsilon_{j+1}} h_{j+1}^{-1}\dots h_{n} x_k^{\epsilon_{n}} h_{n}^{-1}.
\end{align}
Projecting \eqref{eq:minim} in $\mathbb{F}_{ (\mathcal{P}\setminus \{x_k\})}$,
we obtain
\[ 1=\pi^{k}(1)=\pi^{k}(h_1  h_j^{-1} h_{j+1} x_k^{\epsilon_{j+1}} h_{j+1}^{-1}\dots h_{n} x_k^{\epsilon_{n}} h_{n}^{-1})=h_1  h_j^{-1},\]
hence $h_j=h_1$. Back to \eqref{eq:minim}, we get
\[1= h_{j+1} x_k^{\epsilon_{j+1}} h_{j+1}^{-1}\dots h_{n} x_k^{\epsilon_{n}} h_{n}^{-1},\]
from which it also follows that
\[1=h_1 x_k^{\epsilon_1}h_1^{-1}\dots h_j x_k^{\epsilon_j}h_j^{-1}.\]
Since it is not possible that both $h_1^{\epsilon_1}\dots h_j^{\epsilon_j}=1$ and $h_{j+1}^{\epsilon_{j+1}}\dots h_n^{\epsilon_n}=1$, we obtain at least one non-trivial relation, which is of lesser length unless $j=n$. In the case $j=n$, since $\epsilon_j=-\epsilon_1$ and $h_j=h_1$, we obtain the relation of lesser
length
\[
h_2x_k^{\epsilon_2}h_2^{-1} \dots h_{n-1} x_k^{\epsilon_{n-1} } h_{n-1}^{-1}=1, \ h_2^{\epsilon_2}\dots h_{n-1}^{\epsilon_{n-1}}\neq 1,
\]
which is equally absurd. This concludes the proof.
\end{proof}


One can then iterate this procedure, up to the point we obtain
\[  \mathbb{F}_{ \mathcal{P} }\cong
\big( \dots \big( \mathbb{F}_{\{x_1\}} \ltimes \mathbb{F}_{\{x_1,x_2\}}\big)\ltimes \dots  \big)\ltimes \mathbb{F}_{\mathbb{F}_{ (\mathcal{P}-\{x_k\}) }} \big).
\]
We write $\phi_j(g)$ the element on the component  $\mathbb{F}_{\mathbb{F}_{  \{x_1,\dots , x_{j-1} \}}}$, so that
\[g=c_{1} \circ\phi_{1}(g)\dots c_{k} \circ\phi_{k}(g).\]

To be more down-to-earth, we simply wrote $g\in \mathbb{F}_{ \mathcal{P}} $ as
\[g=g' (w_1x_k^{\epsilon_1}w_1^{-1})\dots (w_nx_k^{\epsilon_n}w_n^{-1}), \quad \mbox{with} \quad g',w_1,\dots, w_n \in \mathbb{F}_{ \mathcal{P}\setminus \{x_k\} },\]
 and we then iterated to decompose $g'$ in a similar manner.
As an example, the corresponding writing of $g=x_3x_2x_1x_4x_2x_4^{-1}$ is
\begin{align*} g
&={\color{green!50!blue}x_1 }   ( x_1^{-1}  {\color{green!50!blue} x_2}x_1)
 {\color{green!50!blue} x_2 }
( x_2^{-1} x_1^{-1} x_2^{-1} {\color{green!50!blue} x_3} x_2x_1
x_2)
(x_2^{-1}{\color{green!50!blue}x_4} x_2){\color{green!50!blue} x_4^{-1}}.
\end{align*}

\subsection{Some notation}
For a set $\mathcal{Q}$ (which is either $\mathcal{P}$ or $\mathbb{F}_{ \mathcal{P}\setminus\{x_k\}}$) 
and $g\in \mathbb{F}_{ \mathcal{Q}}$, we define an integer $\ell(g)$ and a finite sequence \[(a_i(g), \alpha_i(g) )_{i\in \{1,\dots, \ell(g) \}} \]
as the unique finite sequence with values in $\mathcal{Q}\times (\mathbb{Z}\setminus\{0\} )$, such that
\[g= \underbrace{a_1(g)\dots a_1(g)}_{\alpha_1(g) }\dots \dots \dots \underbrace{a_{\ell(g)}(g)\dots a_{\ell(g)}(g)}_{\alpha_{\ell(g)}(g) }=\prod_{i=1}^{\ell(g)} a_i(g)^{\alpha_i(g)},\]
and such that
$a_i(g)\neq a_{i+1}(g)$ for all $i\in \{1,\dots, \ell(g)-1\}$. For $i>\ell(g)$, we also set $\alpha_i(g)=0$.


An elementary but useful fact is that for $\mathcal{Q}=\mathbb{F}_{ \mathcal{P}\setminus\{x_k\} }$,
for all $h\in \mathbb{F}_{ \mathbb{F}_{ \mathcal{P}\setminus\{x_k\} }}$, 
\[
c_{k}(h)=\prod_{i=1}^{\ell(h)} a_i(h) x_k^{\alpha_i(h) }a_i(h)^{-1}.
\]

For $x\in \mathcal{Q}$ and $g\in \mathbf{Z}^{\star \mathcal{Q}}$, let $I_{x}(g)$ be the set of the indices $i$ such that $a_i(g)=x$:
\[
I_{x}(g)=\{i\in \{1,\dots, \ell(g)\}, a_i(g)=x\}.
\]
For $k\in \{1,\dots,|I_{x}(g)|\}$, we set $i_{x,k}(g)$ the $k^{\mbox{\scriptsize th}}$ element of $I_{x}(g)$, so that
\[ I_{x}(g)=\{ i_{x,1}(g),\dots,  i_{x,|I_{x}(g)|}(g) \}, \quad  i_{x,1}(g)<\dots< i_{x,|I_{x}(g)|}(g).\]
We also set $\alpha_{x,k}(g)=\alpha_{i_{x,k}(g)}(g)$, and for $k > |I_{x}(g)| $ we set $\alpha_{x,k}(g)=0$. This gives us an ultimately vanishing sequence $\alpha_{x}(g)=(\alpha_{x,k}(g) )_{k\in \mathbb{N}}$, the sequence of the exponents to which $x$ appears in $g$.

\subsection{An inequality }

The following lemma explains in which sense writing $g$ as the product $c_{1}\phi_{1}(g)\dots c_{k}\phi_{k}(g)$ makes it `shorter' than writing it as $a_1(g)^{\alpha_1(g)}\dots a_{\ell(g)}(g)^{\alpha_{\ell(g)}(g)}$.

We endow the set of ultimately vanishing sequences with an order $\preccurlyeq$ obtain as the reflexive and transitive closure of the relation $\mathcal{R}$ given by
\[
u \ \mathcal{R} \ v \iff  \exists i\in \mathbb{N}: \forall j\in \mathbb{N}, u_j=\left\{\begin{array}{ll} v_j & \mbox{if } j<i,\\ v_i+v_{i+1} &\mbox{if } j=i, \\ v_{j+1} &\mbox{if } j>i. \end{array} \right.
\]
In terms of partitions, $u\preccurlyeq v$ corresponds to $v$ being finer than $u$.

For $g\in \mathbb{F}_{ \mathcal{P}}$, if one can find a sequence $u=(u_1,\dots,u_n,0,\dots)$ and $g_1,\dots, g_{n+1}\in \mathbb{F}_{ \{x_1,\dots, x_{j-1}\}  }$ such that $g=g_1x_j^{u_1}g_2 x_j^{u_2}\dots x_j^{u_n} g_{n+1}$,
then $\alpha_{x_j}(g)\preccurlyeq u$.

Besides, with $\|u\|_{l^1}=\sum_{n\in \mathbb{N}} |u_n|$,
\[u\preccurlyeq v\implies \|u\|_{l^1}\leq \|v\|_{l^1}.\]


\begin{lemma}
\label{le:decompo}
  Let $g\in \mathbb{F}_{ \mathcal{P}}$ and $x=x_j\in \mathcal{P}$. Let
  \[
    \pi^{<x}:\mathbb{F}_{ \mathcal{P}}\to \mathbb{F}_{ \{x_1,\dots, x_{j-1} \}} \quad \mbox{and} \quad \pi^{\leq x}:\mathbb{F}_{ \mathcal{P}}\to \mathbb{F}_{ \{x_1,\dots, x_{j} \}}
  \] be the canonical projections.
  Then, $c_j\circ \phi_{j}(g)$ admits the following explicit expression:
  \begin{equation}
  \label{eq:borring}
    c_j\circ \phi_{j}(g)=\prod_{i\in I_{x}(g)} \Big(\pi^{<x}\Big( \prod_{k=i+1}^{\ell(g)} a_k(g)^{\alpha_k(g)} \Big)^{-1} x^{\alpha_i(g)} \pi^{<x}\Big( \prod_{k=i+1}^{\ell(g)} a_k(g)^{\alpha_k(g)} \Big)\Big).
  \end{equation}
\end{lemma}

\emph{The proof of Lemma \ref{le:decompo} consists on first erasing the letters greater than $x_j$, take the right-hand side of \eqref{eq:borring}, and remark that many simplications occur. We omit the full computation, which is elementary but long, not enlightening nor appealing.}

\begin{corollary}
\label{coro:bound}
Let $g\in \mathbb{F}_{ \mathcal{P}}$ and $x=x_j\in \mathcal{P}$.
Then,
\[
\alpha_{x_j}( c_j\circ \phi_j(g) ) \preccurlyeq  \alpha_{x_j}( \pi^{\leq x_j} (g)   )
\preccurlyeq \alpha_{x_j}(g).\]

In particular,
\[
\|\alpha_{x_j}( c_j\circ \phi_j(g) )\|_{l^1}\leq \| \alpha_{x_j}(\pi^{\leq j}(g))\|_{l^1}\leq\| \alpha_{x_j}(g)\|_{l^1}.
\]
\end{corollary}
\begin{proof}
It suffices to apply the previous lemma to $\pi^{\leq x_j}(g)$ instead of $g$.
\end{proof}
Remark that the corollary can also be deduced without using the lemma, and using instead the fact that $c_j\circ \phi_j(g)= (\pi^{<x_j}(g))^{-1}\pi^{\leq x_j}(g)$.

The proper way to use this corollary is the following. Given the set $\mathcal{P}$, choose an order on $\mathcal{P}$ in a particular. Then, instead of studying
$\alpha_{x_j}( c_j\circ \phi_j(g) )$, study instead the sequence $\alpha_{x_j}(\pi^{\leq x_j}(g))$. If $g$ is given as the homotopy class of $\gamma$ in $\pi_1(\mathcal{P})$, then $\pi^{\leq x_j}(g)$ is given as the homotopy class of $\gamma$ in $\pi_1(\{x_1,\dots, x_j )$. This implies that we can forget about all the points $x_k$ with $k>j$ when studying the windings around $x_j$. Of course, the efficiency of this simplification depends strongly on the specific choice of order on $\mathcal{P}$.

When an order is fixed on $\mathcal{P}=\{x_1<x_2<\dots\}$, we will write $\mathcal{P}_{x_j}=\{x_1,\dots, x_j\}$. 

\section{Relations with paths}
\label{sec:relations}
In this section, we let $\mathcal{P}$ be a finite subset of $\R^2\setminus\{0\}$, and $\gamma:[0,1]\to \R^2\setminus\mathcal{P}$ be a continuous function satisfying $\gamma(0)=0$. It is assumed that there is no pair of distinct points $x,y\in \mathcal{P}$ aligned with $0$ (i.e. such that $\widehat{x0y}=0$), and that there is no pair of distinct points $x,y\in \mathcal{P}$ with same norms. One can already think of $\gamma$ as being equal to the Brownian motion $X$, since it is the only curve to which we will apply the results of this section, but these results hold with full generality.

For $s<t\in [0,1]$, we define $\gamma^{s,t}$ as the loop $\gamma^{s,t}=[0,\gamma_s]\cdot \gamma_{[s,t]}\cdot [\gamma_t,0]$. Its class $[\gamma^{s,t}]$ on $\pi_1(\mathcal{P})$ is denoted $g^{s,t}$, and its class on $\pi_1(\mathcal{P}_x)$ is denoted $g^{s,t}_x=\pi^{\leq x}(g^{s,t})$.
We also set $\gamma^t=\gamma^{0,t}$, $g^t=g^{0,t}$ and $g^t_x=g^{0,t}_x$. For $s>t$, we set $g^{s,t}=(g^{t,s})^{-1}$ and $g^{s,t}_x=(g_x^{t,s})^{-1}$.

Remark that $g^{s,t}$ is ill-defined as soon as $\mathcal{P}\cap[0,s]$ or $\mathcal{P}\cap[0,t]$ is non-empty. Since we ultimately want $\gamma$ to be a Brownian motion, such $s$ and $t$ can be numerous, and it is not possible to extend $g^{s,t}$ or even $g^t$, by right or left continuity.

Remark also that $g_{(x)}^{s,t}=g_{(x)}^{s,u}g_{(x)}^{u,t}$, and in particular $g_{(x)}^{s,t}=(g_{(x)}^s)^{-1}g_{(x)}^t$, as soon as the three classes are well-defined.
\subsection{Following the Cayley geodesic along the path}
%
%

 Let us set $h_0=1,\dots, h_N=g^1$ the geodesic walk from $1$ to $g^{1}$ in the Cayley graph $\Gamma$ of $\mathbb{F}_{ \mathcal{P}}$, with respect to the generating family $\mathcal{P}$. In other words, $h_i$ is the prefix of length $i$ of the word
 \[ \prod_{j=1}^{\ell(g)} a_i(g)^{\alpha_i(g)}.\]

 For $i\in \{0,\dots, N\}$,
 we also set $x_i\in \mathcal{P}$ and $\epsilon_i\in \{\pm 1\}$ the unique elements such that $h_i=h_{i-1}x_i^{\epsilon_i}$ and
 \[
 T_i=\inf \{t\in [0,1]: g^t=h_i\}.
 \]
Be careful that, when $\gamma$ is random, the $T_i$ are,  in general, \emph{not} stopping times with respect to the filtration $\mathcal{F}=(\mathcal{F}_t)_{t\in [0,1]}$ generated by $\gamma$: they are only stopping times with respect to the filtration $(\sigma(\mathcal{F}_t,g^1) )_{t\in [0,1]}$ enlarged by the data of $g^1$.

Since $\Gamma$ is a tree and the jumps of $g^t$ are edges on this tree, one has $T_0=0<T_1<\dots<T_N$.

For a finite set $\mathcal{P}$, we define
\begin{equation}
  \label{def:delta}
  \delta(\mathcal{P})=\min \{d(x,y): x,y\in \mathcal{P}\cup \{0\}, x\neq y\}.
\end{equation}
We write $\delta$ for $\delta(\mathcal{P})$ when it is clear what is the underlying set. For $i\in \{1,\dots, N\}$, we set $U_i$ the possibly infinite time
\[
  U_i=\inf \{t>T_i: d(\gamma_t,x_i)\geq \delta\}.
\]
\begin{lemma}
\label{le:Uexists}
Let $i<k$ be such that $x_i=x_k$. Assume that there exists $j$ such that $i<j<k$ and and $x_j\neq x_{i}$. Then, $U_i<T_{k}$.
\end{lemma}
\begin{proof}
Let $l=\min\{ n>i: x_n\neq x_i\}-1$ and $m=\min \{n>l: x_n=x_i\}$. Then, $l\geq i$ and $m\leq k$, so that it sufficies to show $U_{l}<T_m$.
Let $T_{l}^+> T_{l}$ and $T_{m}^-< T_{m}$ be such that  
$g^{T_{l}^+}=h_{l}$ and $g^{T_m^-}=h_{m-1}$.

Then $g^{T^+_l,T^-_l}=h_l^{-1}h_{m-1}\neq 1$. Since $x_n\neq x_i$ for all $n\in\{l+1,\dots,m-1\}$, the word
\[ x_{l+1}^{\epsilon_{l+1}}\dots x_{m-1}^{\epsilon_{m-1}}\]
representing $h_l^{-1}h_{m-1}$, and hence different from $1$, is also the word representing the homotopy class of $\gamma^{T^+_l,T^-_l}$ in $E=\R^2\setminus(\mathcal{P}\setminus \{x_i\})$. It follows that $\gamma^{T^+_l,T^-_l}$ is non-contractible in $E$.
As $T^+_l$ approaches $T_l$ and $T_m^-$ approaches $T_m$, both $\gamma_{T^+_{l}}$ and
$\gamma_{T^-_{m}}$ approaches the axe passing through $0$ and $x_i$. Therefore, the triangle with vertices $0$, $T^+_l$ and $T^-_m$ becomes infinitely thin, up to some point when it does not contain any points of $\mathcal{P}\setminus \{x_i\}$. Then, the loop $\gamma^{T^+_l,T^-_l}$ can be continuously deformed in $E$ into $\gamma_{|[T^+_l,T^-_{m}]}\cdot [\gamma_{T^-_{m} },\gamma_{T^+_l}]$, and then into $\gamma'=\gamma_{|[T_l,T_{m}]}\cdot [\gamma_{T_{m} },\gamma_{T_l}]$.
Since $\gamma^{T^+_l,T^-_l}$ is not contractible in $E$, $\gamma'$ is also non-contractible in $X$. In particular, since the open ball $B(x_i,\delta)$ is included on $E$, and therefore contractible, $\gamma'$ cannot remain on $B$.
Thus, either $\gamma_{|[T_l,T_{m}]}$ or $[\gamma_{T_{m} },\gamma_{T_l}]$ is not included on $B(x_i,\delta)$. In both case, there exists $s\in [T_l,T_m]$  such that $\gamma_s\notin B(x_i,\delta)$, so that $U_i<T_m$.
\end{proof}

\subsection{Half-turns}
For $x\in \mathcal{P}$, we now define an integer $\theta_{\frac{1}{2}}(x,\gamma)$, the \emph{number of half-turns of $\gamma$ around $x$}.
Let $d_x^1$ and $d_x^2$ be the two half-lines delimited by $x$ and orthogonal to the vector from $0$ to $x$ (see Figure \ref{fig:dessin} below).
Let $t_0=0$ and $t_1$ be the (possibly infinite) first time $\gamma$ hits $d_x^1$. Times $t_2,t_3,\dots$ are then defined recursively by the formulas
\begin{equation}
t_{2i}=\inf \{t>t_{2i-1}: \gamma_t\in d_x^2\},\ t_{2i+1}=\inf \{t>t_{2i}: \gamma_t\in d_x^1\}.
\label{eq:def:ti}
\tag{*}
\end{equation}
Only finitely many of these times are less than $1$, after which they are all infinite. The integer $\theta_{\frac{1}{2}}(x)$ is then defined as the maximal index $i$ such that $t_i$ is finite, plus $1$. This additional $1$ is to account for the potential winding of $\gamma$ before it reaches $d_x^1$ for the first time. 
 %

B
 In the following, the set $\mathcal{P}$ is ordered by the norm of its elements,
 \[ x<y\iff \|x\|_2<\|y\|_2.\]
 Since we assumed that all the points of $\mathcal{P}$ have different norms, this order is total.

The key idea of this paper is that when the element of $\mathcal{P}$ is endowed with this order, the number of half-turns $\theta_{\frac{1}{2}}(x,\gamma)$ offers an upper bound on the $\|\alpha_x(\pi^{\leq x}( [\gamma]))\|_{l^1}$,
as the following proposition shows, whilst being easy to control when $\gamma$ is a Brownian motion.

\begin{proposition}
  \label{prop:bound12}
  The following inequality holds.
  \[ \|\alpha_x( g^1_x  )\|_{\ell^1} \leq \theta_{\frac{1}{2}}(x,\gamma).\]
\end{proposition}
\begin{proof}
  Since $g^1_x=g_x^{0,t_1}g_x^{t_1,t_2}\dots
  g_x^{t_{n-1},t_n}g_x^{t_n,1}$ with $n=\theta_{\frac{1}{2}}(x, \gamma)$, and since
  \[\|\alpha_x(gg')\|_{\ell^1}\leq \|\alpha_x(g)\|_{\ell^1}+ \|\alpha_x(g')\|_{\ell^1},\]
it suffices to prove that $\| \alpha_x(g^{t_i,t_{i+1} }_x)\|_{\ell^1}$ cannot be larger than $1$. We advise the reader to convince herself or himself that this holds before looking at our proof.

  We assume that $i$ is odd, the other case being dealt with similarly. Then, $\gamma_{t_i}$ lies in $d^1_x$. Set
  \[
  s_i=\sup\{t<t_{i+1}: \gamma_s\in d^1_x\}.
  \]
  Then, $\gamma_{|[t_i,s_i]}$ takes value  on $\R^2\setminus d^2_x$. Let $B$ be the closed ball of radius $\|x\|-\epsilon$, with $\epsilon$ small enough for $B$ to contain $\mathcal{P}_x\setminus \{x\}$. Then, $B'=B\setminus\mathcal{P} $ is a deformation retract of
  $\R^2\setminus (\mathcal{P}_x \cup  d^2_x)$, so that we can continuously deform the loop $\gamma^{t_i,s_i}$, in $\R^2\setminus (\mathcal{P}_x \cup  d^2_x)$ and hence $\R^2 \setminus \mathcal{P}_x $,  into a loop $\gamma'$ based at $0$ and with values in $B'$.
  It follows that the class of $\gamma_{|[t_i,s_i]}$ in $\pi_1(\mathcal{P}_x)$ and in $\pi_1(\mathcal{P}_x\setminus \{x\})$ are equal, and therefore $\alpha_x(g^{t_i,s_i}_x)=0$.

  Let us now look at $g^{s_i,t_{i+1}}$. The path $\gamma_{[s_i,t_{i+1}] }$ is contained in one of the two closed half-space $H,H'$ delimited by the line $d^1_x\cup d^2_x$. If if lies on $H$, one can repeat the argument above, and deduce that  $\alpha_x(g^{s_i,t_{i+1}}_x)=0$. Otherwise, since there is no point in $\mathcal{P}_x \cap H'$ but $x$,   the class $g^{t_i,s_i}_x$ does not depend on $\gamma_{[s_i,t_{i+1}] }$ but only on its endpoints $\gamma_{s_i}, \gamma_{t_i}$.
  One can therefore replace  $\gamma_{[s_i,t_{i+1}] }$ with any curve with the same endpoints and staying inside $H'$. Replacing it with a curve $\gamma'$ with monotonic angle around $0$, we deduce that
$\| \alpha_x(g^{s_i,t_{i+1} }_x)\|_{\ell^1}=1$, which concludes the proof.
\end{proof}

\begin{figure}
\includegraphics{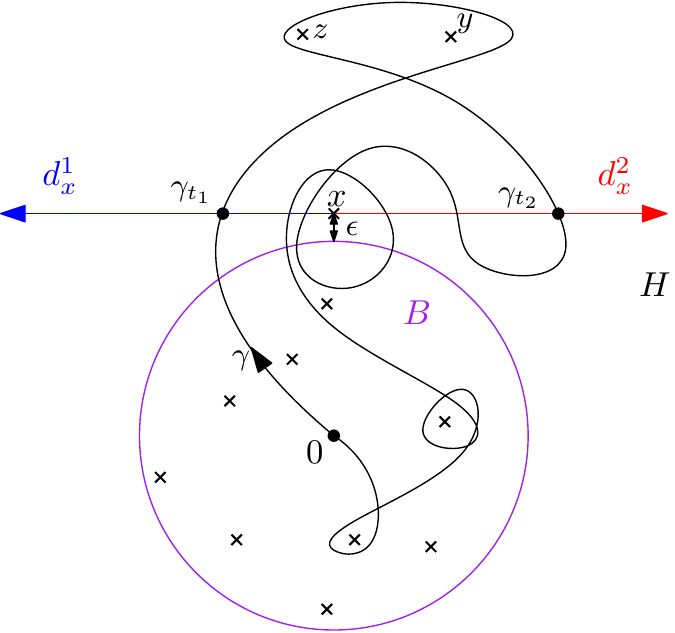}
\caption{\label{fig:dessin} There is no bound on the number of time the letter $x$ appears in $g^{t_i,t_{i+1}}$, but it can appear at most once in $g^{t_i,t_{i+1}}_x$. In this case, the class of $\gamma^{t_1,t_2}$ in $\pi_1(\mathcal{P} )$ is  $w_1xw_2y^{-1}w_3x^{-1}w_4z^{-1}xw_5$, with $x,y,z$ not appearing in the $w_i$. Its class in $\pi_1(\mathcal{P}_x )$
is simply $w'_1xw'_2$.}
\end{figure}


Let $x\in \mathcal{P}$, $g\in \mathbb{F}_{ \mathcal{P}}$. Let us recall that there is a canonical writing of $g$
\[ g=\prod_i a_i(g)^{\alpha_i(g)}, \]
and that $I_x(g)$ is the set of the indices $i$ such that $a_i(g)=x$.

For a positive integer $i$, we now define $S^{(i)}(x,g)$ as the sum of all but the $i-1$ largest values among the $(|\alpha_i(g)|)_{i\in I_x(g)}$. In particular, $S^{(1)}(x,g)= \| \alpha_x(g)\|_{\ell^1}$.
\begin{definition}
Write $I_x(g)$ as $\{i_1<\dots <i_{\# I_x(g)}\}$. For $k\in \{1,\dots ,\#I_x(g)\}$, let $u_k= \alpha_{i_k}(g)$.
Let $\pi$ be a permutation of $\{1,\dots, \# I_x(g)\}$ which is such that $|u_{\pi(1)}|\geq |u_{\pi(2)}|\geq \dots |u_{\pi(\# I_x(g))}|$. Then, we set
\[
  \beta_k(x,g)=u_{\pi(k)} \quad \mbox{and} \quad  S^{(i)}(x,g)=\sum_{k=i}^{\# I_x(g)} |u_{\pi(k)}|=\sum_{k=1}^{\# I_x(g)}|\alpha_{i_k}(g)|-\sum_{k=1}^{i-1} |\beta_i(x,g)|.
\]
\end{definition}
\begin{remark}
Because of the absolute values, the sequence $\beta_i(x,g)$ might not uniquely determined by $g$. One can for example enforce the condition \[ |\beta_i(x,g)|=|\beta_j(x,g)|, \beta_i(g)\neq\beta_j(x,g), i<j \implies \beta_i(x,g)>\beta_j(x,g) \]
in order to get unicity. The peculiar choice we make does not play a specific role in the following.
\end{remark}

Before we actually study the case when $\gamma$ is a Brownian motion, let us conclude this section with the following elementary fact.
\begin{lemma}
\label{le:six3}
Let $u_1,\dots, u_k$ be a finite sequence of positive real numbers and $i$ a positive integer.

Let $\pi$ be a permutation of $\{1,\dots, k\}$ which is such that $u_{\pi(1)}\geq u_{\pi(2)}\geq \dots u_{\pi(k)}$, and let \[S^{(i)}=\sum_{n=i}^k u_{\pi(i)},\]
that is $S^{(i)}$ is the sum of all but the $i-1$ largest elements of the sequence.

Then, there exists
$j_1=0<j_2<\dots< j_{i+1}=k$ such that for all
$l\in \{1,\dots, i\}$,
\[\sum_{n=j_l+1}^{j_{l+1}} u_n \geq \left\lfloor \frac{ S^{(i)}}{i} \right\rfloor.\]
\end{lemma}
The proof is left to the reader.
%
% ...

%
%
%
%

Finally, one can deduce the following proposition.
\begin{proposition}
\label{prop:Si}

Let $x\in \mathcal{P}$, and $i,N$ be two positive integers. Assume that $S^{(i)}(x,g^1_x)\geq N$.

Then, there exists $0<u_1<\dots<u_i= 1$ such that for all $k\in \{1,\dots, i-1\}$,
\[\theta_{\frac{1}{2}}(x,\gamma^{u_k,u_{k+1}})\geq \left\lfloor \tfrac{N}{i} \right\rfloor \quad \mbox{ and }\quad
d(\gamma_{u_{k}},x)\geq \delta.\]
\end{proposition}
\begin{proof}
We will show that there exists $0<t_1<s_1<u_1<t_2<\dots<u_{i-1}<t_i<s_i=1$ such that
\begin{itemize}
\item For all $k\in \{1,\dots, i\}$,
$\theta_{\frac{1}{2}}(x,\gamma^{t_k,s_{k}})\geq \left\lfloor \tfrac{N}{i} \right\rfloor$,
\item For all $k\in \{1,\dots, i-1\}$,
$d(\gamma_{u_{k}},x)\geq \delta$.
\end{itemize}
Since $\theta_{\frac{1}{2}}(x,\gamma^{t,s})$ is monotonic in $s$ and $t$, the two properties then still holds with  $s_k$ replaced by $u_k$ and $t_k$ replaced by $u_{k-1}$, which allows to conclude.

Let $i_1(x),\dots , i_{\#I_x(g^1_x)}(x)$ be the elements of $I_x(g^1_x)$

We first apply Lemma \ref{le:six3} to find $j_1=0<\dots<j_{i+1}=\#I_x(g^1_x) $ such that for all $l\in \{1,\dots, i\}$,
\[
\sum_{n=j_l+1}^{j_{l+1}} |\alpha_{i_n(x)}(g^1_x )|\geq \left\lfloor \tfrac{N}{i} \right\rfloor.
\]

For $l\in \{1,\dots, i+1\}$, let
\[ k_l= \sum_{i=1}^{j_l} |\alpha_i(g^1_x)|, \quad k'_l= \sum_{i=1}^{j_l+1} |\alpha_i(g^1_x)|
\]
so that
\[
g_x^{T_{k_l}}=\prod_{i\leq j_l} a_i(g^1_x)^{\alpha_i(g^1_x)}, \quad
g_x^{T_{k'_l}}=\prod_{i\leq j_l+1} a_i(g^1_x)^{\alpha_i(g^1_x)}.
\]
Then $S^1(x, g_x^{T_{k'_l}, T_{k_{l+1}}})=\sum_{n=j_l+1}^{j_{l+1}} |\alpha_{i_n(x)}(g^1_x)|\geq \left\lfloor \tfrac{N}{i} \right\rfloor $.

We apply Proposition \ref{prop:bound12} to deduce that
$\theta_{\frac{1}{2}}(x, \gamma^{T_{k'_l},T_{k_{l+1}}})\geq \left\lfloor \tfrac{N}{i} \right\rfloor $.

Thus, we use $t_l=T_{k'_l}$ and $s_l=T_{k_{l+1}}$.

As for $u_l$, its existence is ensured by Lemma \ref{le:Uexists} using the fact that $a_{i_n(x)+1}(g^1_x)\neq x$ for all $n$, hence in particular for $n=j_l$.
\end{proof}

\section{The case of Brownian paths}
\label{sec:brown}
\subsection{Minimal spacing in a Poisson set}
It is now time to apply some probability theory. First, we need to get some bound on $\delta$ the minimal distance between two points in $\mathcal{P}_K$. Since $\mathcal{P}_K$ is infinite, $\delta(\mathcal{P}_K)$ is actually $0$. Fortunately, one can freely replace $\mathcal{P}_K$ with some large ball containing the Brownian motion $X$.

With this in mind, we set $F_R$ the event
\[ F_R: \|X\|\leq R,\]
where $\|X\|$ is the maximum of $t\mapsto \|X_t\|_2$.
and
\[\mathcal{P}_K^R=\mathcal{P}_K\cap B(0,R).\]

For a given $R$, the cardinal of $\mathcal{P}_K^R$ is of order $\pi R^2 K$, and we will see that $\delta(\mathcal{P}_K^R)$ is of order at least $K^{-1}$, so we also set
\[
E_R: \# \mathcal{P}_K^R \leq 4R K\log(K) \mbox{ and } \delta(\mathcal{P}_K^R)\geq (K\log(K))^{-1}.
\]
We could have replace $K\log(K)$ with any function $f$ such that $K=o(f(K))$ and $f(K)=K^{1+o(1)}$, the point here is that the event $E_R$ has probability arbitrary close to $1$ when $K$ goes to infinity, as the following lemma shows.

%
%
\begin{lemma}
\label{lemma:delta_small}
\[\mathbb{P}\big(\delta(\mathcal{P}^R_K) \leq (K \log(K))^{-1} \big)\underset{K\to +\infty}\to 0.\]
\end{lemma}
\begin{proof}

We cover $B(0,R)$ with $\lfloor C (K\log(K))^{2} \rfloor$ balls $B_1,\dots$ of radius $(K\log(K))^{-1}$ (with $C$ a constant depending only on $R$). This can be done for example by choosing the points on a properly scaled grid. Let then $B'_1,\dots, $ be the balls with the same centers $x_1,\dots$ and radii $2(K\log(K))^{-1}$. Then, for any two points $x,y\in B(0,R)$ at distance less than $(K\log(K))^{-1}$ from each other,
there exists some $i$ such that both $x$ and $y$ lies on $B'_i$. One can take, for example, $i$ such that $x\in B_i$, since in that case
\[d(y,x_i)\leq d(y,x)+d(x,x_i)\leq 2   (K\log(K))^{-1}.\]
Thus, the considered event is included on the event
\[ \exists i\in \{1,\dots , \lfloor C (K\log(K))^{2} \rfloor\}: \# (\mathcal{P}_K\cap B'_i)\geq 2 \mbox{ or } \exists x\in \mathcal{P}_K\cap B(0,(K\log(K))^{-1}).\]
For a given $i$, $\#   (\mathcal{P}_K\cap B'_i)$ is a Poisson variable with intensity $K\times (K\log(K))^{-2}=K^{-1}\log(K)^{-2} $, the probability that it is at least $2$ is of order $K^{-2}\log(K)^{-4}$, so that
\begin{align*}
\mathbb{P}\Big( &\exists i\in \{1,\dots , \lfloor CR (K\log(K))^{2} \rfloor\}: \# (\mathcal{P}_K\cap B'_i)\geq 2\Big)\\
&\leq C'  (K\log(K))^{2} K^{-2}\log(K)^{-4}=C' \log(K)^{-2} \underset{K\to \infty}{\to} 0.
\end{align*}

It is left to the reader to show that the probability of the event $\exists x\in \mathcal{P}_K\cap B(0, (K\log(K))^{-1})$ is small as well.
\end{proof}


\subsection{Controlling the $\beta_i$}
It is now time to use the fact that our path $X$ is Brownian.
Let us recall that we denote by $\theta(x)$ the winding number of $X$ around $x$, which counts algebraically the number of turns of $X$ around $x$, whilst $\theta_{\frac{1}{2}}(x)$ counts non-algebraically the number of half-turns of $X$ around $x$. We define $\theta_{s,t}(x)$ as the winding number of $X_{|[s,t]}$ around $x$.
Working with a Brownian motion allows for two things. From one side, it allows to easily relates $\theta$ and $\theta_{\frac{1}{2}}$, as we will show on the next subsection. From the other side, it allows to use some already established estimation about $\theta$.

We will use the following idea. For the Brownian motion to wind a large amount of time around a given point $z$, the Brownian motion has to go extremely close to $z$. The windings are then mostly due to the small fraction of time it spend close to $z$. Since it is highly unlikely that the Brownian motion goes close to $z$ twice, the windings are actually mostly due to a very small \emph{interval} of time.

In particular, if we wait for the Brownian motion to wind a lot around $z$, and then to go far\footnote{In our case, `far' is actually rather close, but not `extremely close'.} from $z$, it is unlikely that it will winds a lot around $z$ after that. The following proposition gives a way to quantize this probability.
\begin{lemma}
\label{le:ImprovedZhan}
Let $n$ be  positive integer. Then, there exists a constant $C$ such that for all $r\leq 1$, for all $z\in \R^2$ with $\|z\|\geq r$,
for all positive integers $N_1,\dots, N_n$,
\begin{align*}
\mathbb{P}\Big(&\exists s_1<t_1<u_1<s_2<\dots<u_{n-1}<s_n<t_n<1: \\
&\forall i\in \{1,\dots,n\}, |\theta_{s_i,t_i}(x)|\geq N_i,
\forall i\in \{1,\dots,n-1\}, d(X_{u_i},x)\geq r\Big) \leq C \frac{ \log(r^{-1} \max_{i}\log(N_i) )^n}{N_1\dots N_n}.
\end{align*}
\end{lemma}
\begin{proof}
  At page 117 of \cite{zhan}, we can found an inequality on the maximal possible value of the winding function $\theta(x,X_{|[0,t]})$ as $t$ varies. With our notation and a very small amount of work, we deduce from this inequality that, provided $r^{-2}\log(N)$ is large, which is our case,
  \begin{equation}
  \label{eq:doubleZhan}
  \sup_{x: \|x\|\geq r} \mathbb{P}( \sup_{t\in [0,1]} \theta(x,X_{|[0,t]})\geq N   )\leq C \frac{\log(r^{-2} \log(N)) }{N}.
  \end{equation}


  We directly deduce that
  \[
  \sup_{x: \|x\|\geq r} \mathbb{P}( \sup_{s<t\in [0,1]} \theta(x,X_{|[s,t]})\geq N   )\leq C' \frac{\log(r^{-2} \log(N))}{N}.
  \]

  It then suffices to use the Markov property at the stopping times
  \[T_0=0,\quad T_{i+1}=\inf \{t>U_i:\exists s\in [U_i,t]\ s.t.\ |\theta_{s,t}(x)|\geq N_i\},\quad U_i=\inf\{u>T_i: d(X_{u},x)\geq r\}.\]
  For the event we study to happens, the times $T_1,\dots,T_n$ must all be finite. We conclude by using the inequality \eqref{eq:doubleZhan} $n$ times in a row.
\end{proof}

\begin{corollary}
\label{coro:improvedZhan}
Let $n$ be a positive integer and $R,\epsilon>0$. Then, there exists a constant $C$ such that for all $N_1,\dots, N_n$,
\begin{align*}
\mathbb{P}\Big(&  E_R \cap  \exists x\in \mathcal{P}_K^R , \exists s_1<t_1<u_1<s_2<\dots<u_{n-1}<s_n<t_n<1: \forall i\in \{1,\dots,n\},\\
& |\theta_{s_i,t_i}(x)|\geq N_i,
\forall i\in \{1,\dots,n-1\}, d(X_{u_i},x)\geq \delta \Big) \leq C \frac{ K \log(K)^{n+1} \max_i \log(\log(N_i))} {N_1\dots N_n}.
\end{align*}
\end{corollary}
\begin{proof}
We enumerate $\mathcal{P}_K^R=\{p_1,\dots, p_{\#\mathcal{P}_K^R} \}$ in a way which is independent from $X$. For $i>\#\mathcal{P}_K^R$, we set $p_i=p_1$ as a convention.

Then, for any deterministic $i\in \{1,\dots,\lfloor K\log(K)\rfloor \}$, we can apply the previous lemma to $x=p_i$ and with $r=\min(\delta(\mathcal{P}_K),1)$.

The conclusion then follows from summation over $i$, using the fact that for $K$ large enough, \[K\log(K) \log(K\log(K))^n\leq 2 K\log(K)^{n+1}.\]
\end{proof}
%

One can finally give some controls on the $\beta_i$. The following corollary allows to exclude three kind of events. The first convergence allows to exclude the possibility that the same letter $x$ appears once with a very large exponent, and then once again with a large exponent. The second convergence allows to exclude the possibility that the same letter $x$ appears three times, each time with a large exponent. The third convergence allows to exclude the possibility that many letters $x$ each appears twice, each time with a large exponent.
\begin{corollary}
\label{coro:threeBounds}
Let 
$\epsilon\in(0,\tfrac{1}{6})$. Then,
\[
\mathbb{P}\Big(\exists x\in \mathcal{P}_K^R: \beta_1(x,g_x^1)\geq K^{\frac{2}{3}}, \beta_2(x,g_x^1)\geq K^{\frac{1}{2}-\epsilon})\underset{K\to \infty}\to 0
\]
and
\[
\mathbb{P}\Big(\exists x\in \mathcal{P}_K^R: \beta_2(x,g_x^1)\geq K^{\frac{1}{2}+\epsilon}\Big)\underset{K\to \infty}\to 0.
\]
Besides,
\[
\mathbb{P}\Big(\#\big\{ x\in \mathcal{P}_K^R : \beta_2(x,g_x^1)\geq K^{\frac{1}{2}-\epsilon}\big\} \geq K^{3\epsilon}   \Big)\underset{K\to \infty}\to 0.
\]
\end{corollary}
\begin{proof}
In the three cases, we control the probability of the given event intersected with $E_R\cap F_R$. Since $\mathbb{P}(E_R\cap F_R)$ can be made arbitrary close to $1$, it is sufficient to conclude.
Let $G_1$ be the event
\[\exists x\in \mathcal{P}_K^R: \beta_1(x,g_x^1)\geq K^{\frac{2}{3}}, \beta_2(x,g_x^1)\geq K^{\frac{1}{2}-\epsilon}.\]
On the event $G_1\cap E_R\cap F_R$,  let us write $g_x^1$ as $g_1 x^{\pm \beta_1 } g_2 x^{\pm \beta_2} g_3$ (the case $g_x^1=g_1 x^{\pm \beta_2 } g_2 x^{\pm \beta_1} g_3$ is similar). Let $h_0,\dots, h_N$ be the geodesic path from $1$ to $g^1_x$, and let $i,j,k,l$ be the indices such that
$h_i= g_1$, $h_j=g_1x^{\pm \beta_1 } $, $h_k=g_1 x^{\pm \beta_1 } g_2$, $h_l=g_1 x^{\pm \beta_1 } g_2 x^{\pm \beta_2} $, and let $T_i$ be the infimum of the times $t$ such that the class of $X_{0,t}\cdot [X_{t},X_0]$ on $\mathcal{P}_x$ is equal to $h_i$.
Then, the loop $X_{|[T_i,T_j]}\cdot [X_{T_j},X_{T_i}]$ winds $\pm \beta_1$ times around $x$, and the loop  $X_{|[T_k,T_l]}\cdot [X_{T_l},X_{T_k}]$ winds $\pm \beta_2$ times around $x$.

According to Lemma \ref{le:Uexists}, there exists $U\in [T_j,T_k]$ such that the distance between $X_{U}$ and $x$ is at least $\delta$. Therefore, the considered event implies
\[\exists x\in \mathcal{P}_K,
\exists s_1<t_1<u_1<s_2<t_2: |\theta_{s_1,t_1}(x)|\geq K^{\frac{2}{3}}, |\theta_{s_2,t_2}(x)|\geq K^{\frac{1}{2}-\epsilon}, d(X_{u_1},x)\geq \delta,
\]
or the similar event with the exponents $\frac{2}{3}$ and $\frac{1}{2}-\epsilon$ switched.
We then apply Corollary \ref{coro:improvedZhan}, and we get a probability smaller than $CK \log(K)^{4} K^{-\frac{2}{3}}K^{-\frac{1}{2}+\epsilon}$, which goes to $0$ as $K\to \infty$ since $\epsilon<\frac{1}{6}$.
This proves the first convergence.

The second convergence is obtain in a similar manner, we left it to the reader.

For the third one, we enumerate $\mathcal{P}^R_K=\{x_1,\dots, x_{\# \mathcal{P}^R_K}\}$ as before, in any way independent from $X$, and we set conventionally $x_i=x_1$ for $i>\# \mathcal{P}^R_K$.


For a given $i\in \mathbb{N}$, let $H_i$ be the event
\[ \beta_2(x_i,[\bar{X}])\geq K^{\frac{1}{2}-\epsilon}.\]
Applying, as before, Lemma \ref{le:Uexists} and then Lemma \ref{le:ImprovedZhan} (instead of Corollary \ref{coro:improvedZhan}), we obtain that
\[
\mathbb{P}(F_R\cap E_R \cap H_i )\leq C \log(K)^2 K^{-1+2\epsilon}.
\]

Therefore,
\[
\mathbb{E}[\# \{ i:  H_i \mbox{ holds }\} \mathbbm{1}_{F_R\cap E_R} ] \leq C K \log(K) \log(K)^2 K^{-1+2\epsilon}=C\log(K)^3 K^{2\epsilon}.
\]
From Markov inequality, we deduce that
\[
\mathbb{P}\big(\# \{ i:  H_i\cap F_R\cap E_R \mbox{ holds }\} \geq K^{3\epsilon} \big) \leq C\log(K)^3 K^{-\epsilon}\underset{K\to \infty}\to 0.
\]
This concludes the proof.
\end{proof}
%
%
%
%
%
%
%
The previous bounds tells us about the highest values $\beta_1,\beta_2,\beta_3$, but it does not tell us anything about the behaviour of the tail $S^{(i)}$. Controlling this tail is the goal of the following subsection.

\subsection{Controlling the tails $S^{(i)}$}
The idea is to apply Proposition \ref{prop:Si}, which relates $S^{(i)}$ with $\theta_{\frac{1}{2}}$. We first show that $\theta_{\frac{1}{2}}$ can itself be related to $\theta$ when the curve is Brownian.

\begin{lemma}
  \label{le:sqrt}
  There exists a finite constant $C$ such that for all $N\geq 1$, and $x\in \R^2$,
  \[ \mathbb{P}(\theta_{\frac{1}{2}}(x)\geq N )\leq C  \mathbb{P}(\theta(x)\geq \sqrt{N} ).\]
\end{lemma}
\begin{proof}
In short, it follows from the fact that $\theta(x)$ is, up to an error of $\pm 1$, the sum of $\theta_{\frac{1}{2}}(x)$ i.i.d. centered Bernoulli variables.

Let use recall from \eqref{eq:def:ti} that the stopping times $t_i$ correspond to the times $t$ when $\theta_{\frac{1}{2}}(x, X_{|[0,t]})$ increases. It is easily seen that the real-valued windings $\tilde{\theta}(x, X_{|[0,t_i]})$ and $\tilde{\theta}(x, X_{|[0,t_{i+1}]})$
are related  by
\[ \epsilon_i=\tilde{\theta}(x, X_{|[0,t_{i+1} ]})-\tilde{\theta}(x, X_{|[0,t_{i}]})\in \{-\tfrac{1}{2}, \tfrac{1}{2} \}. \]
  The integer winding number $\theta(x)$ is then given by
  \[\theta(x)=\sum_{i=1}^{\theta_{\frac{1}{2}}(x) } \epsilon_i+\epsilon_0,\]
  with $\epsilon_0\in \{-1,-\frac{1}{2},0,\frac{1}{2},1\}$.

Since the $t_i$ are stopping times, we can apply the reflection property of the Brownian motion at the times $t_i$, with reflection around the axe $d^1_x\cup d^2_x$. We can deduce from it that the variables $\epsilon_i$ are i.i.d. symmetric Bernoulli variables, and that for any $n$, $(\epsilon_1,\dots, \epsilon_n)$ is independent from $\theta_{\frac{1}{2}}(x)$ conditional to $n\leq \theta_{\frac{1}{2}}(x)$.
  Hence,
  \[
  \mathbb{P}\big(\theta(x)\geq \sqrt{N} \big|\theta_{\frac{1}{2}}(x)= N \big)=
  \mathbb{P}\Big(\epsilon_0+ \sum_{i=1}^{\theta_{\frac{1}{2}}(x) } \epsilon_i\geq \sqrt{N}\Big|\theta_{\frac{1}{2}}(x)=N\Big)\underset{N\to \infty}{\longrightarrow} \int_{1}^{+\infty} \frac{e^{-2 t^2 }}{\sqrt{\pi/2}}\d t\neq 0.
  \]
  Since the right-hand side is stricly positive for all $N$, there exists a constant $c>0$ such that for all $N$,
  \[ \mathbb{P}\big(\theta(x)\geq \sqrt{N} \big|\theta_{\frac{1}{2}}(x)= N \big)\geq c.\]

  For any two real random variables $U,V$,
  \[\mathbb{P}(U\geq u)\leq \frac{\mathbb{P}(V\geq v ) }{\mathbb{P}\big( V\geq v\big| U\geq u \big) }.\]
  In particular,
  \[
  \mathbb{P}(\theta_{\frac{1}{2}}(x)\geq N )\leq \frac{\mathbb{P}(\theta(x)\geq \sqrt{N} ) }{ \mathbb{P}\big(\theta(x)\geq \sqrt{N} \big|\theta_{\frac{1}{2}}(x)= N \big) }
  \leq c^{-1}\mathbb{P}(\theta(x)\geq \sqrt{N} ),
  \]
  which concludes the proof.
\end{proof}

One can now give the following bound on the tails $S^{(i)}$.
\begin{lemma}
  \label{le:s5}
  Let $\epsilon\in (0,\tfrac{1}{10})$. Then,
  \[
  \mathbb{P}( \exists x\in \mathcal{P}_K: S^{(5)}(x,g^1_x )\geq K^{\tfrac{1}{2}-\epsilon}) \underset{K\to \infty}\to 0
  \]
\end{lemma}
At this stage, the reader should starts to understand the point of this section and the previous one: this lemma, which will follow from our work in these sections, make use somehow closer to the `much simpler problem' we solved in section \ref{sec:simpler}. Instead of dealing with the whole apparitions $x$ in $[\bar{X}]$, we will only have to deal with the apparitions that comes with a large exponent.
\begin{proof} 
Set $N= \tfrac{K^{\frac{1}{2}-\epsilon }}{5}$.
For a given point $x\in \R^2$ with $\|x\|\geq \delta$,
let $U_1=0$, $U_{i+1}=\inf \{t:  \theta_{\frac{1}{2}}(x,X^{U_i,t} )\geq N, d(X_{t},x)\geq \delta$. Then, using Proposition \ref{prop:Si}, we have
\[
\mathbb{P}(S^{(5)}(x,g^1_x )\geq K^{\frac{1}{2}-\epsilon } )\leq \mathbb{P}\big( U_5 <1\big).
\]
Since the $U_i$ are stopping times, one can apply the Markov property to deduce
\[ \mathbb{P}\big( U_5 <1\big)\leq  \mathbb{P}(U_1<1)^5.\]
Lemma \ref{le:sqrt} says that $ \mathbb{P}(U_1<1)\leq C \mathbb{P}(\theta(x) \geq \sqrt{N} )$, which thanks to Equation \eqref{eq:doubleZhan} we know to be less than $C'\frac{\log(\delta^{-2} \log(N))}{N^{\frac{1}{2}}}$.

Combining these inequalities together, we get
\[
\mathbb{P}(S^{(5)}(x,g^1_x )\geq K^{\frac{1}{2}-\epsilon } )\leq C (K^{\frac{1}{2}-\epsilon })^{-\frac{5}{2}} \log(\delta^{-2} \log(K))^5.
\]

Since this holds for all $\|x\|\geq \delta$,
we get
\[ \mathbb{P}( E_R\cap F_R\cap \exists x\in \mathcal{P}^R_K:S^{(5)}(x,g^1_x )\geq K^{\frac{1}{2}-\epsilon }  )
\leq C' RK\log(K)^6 K^{-\frac{5}{4}+\frac{5}{2}\epsilon }\underset{K\to \infty}\longrightarrow 0.\]

\end{proof}

This estimation in itself is not very practical, but it can then be improved in the three following ways, which correspond to the three convergences of Corollary \ref{coro:threeBounds}.

\begin{proposition}
  \label{prop:threeBounds}
  Let $\epsilon\in (0,\tfrac{1}{20})$. Then,
  \[
  \mathbb{P}\big( \exists x\in \mathcal{P}_K: S^{(2)}(x,g^1_x)\geq K^{\frac{1}{2}-\epsilon} \mbox{ and } \beta_1(x,g^1_x)\geq K^{\frac{2}{3}}\big)\underset{K\to \infty}\to 0,
  \]
  \[
  \mathbb{P}\big( \exists x\in \mathcal{P}_K: S^{(2)}(x,g^1_x)\geq K^{\frac{1}{2}+\epsilon}\big)\underset{K\to \infty}\to 0,
  \]
  \[
  \mathbb{P}\big( \# \{x\in \mathcal{P}_K: S^{(2)}(x,g^1_x)\geq K^{\frac{1}{2}-\epsilon}  \}\geq K^{5\epsilon}  \big)\underset{K\to \infty}\to 0.
  \]
\end{proposition}
\begin{proof}
  The three proofs are similar, we only write the first one.
  Thanks to Lemma \ref{le:s5}, we can assume that $S^{(5)}(x,g^1_x)\leq K^{\frac{1}{2}-2\epsilon}$ for all $x\in \mathcal{P}_K$.
  Then, for $S^{(2)}(x,g^1_x)$ to be larger than $K^{\frac{1}{2}-\epsilon}$,
  $\beta_2(x,g^1_x)$ has to be larger than $\frac{K^{\frac{1}{2}-\epsilon} -K^{\frac{1}{2}-2\epsilon} }{3}$,
  which is larger than $K^{\frac{1}{2}-\frac{3 \epsilon}{2}} $ for $K$ large enough. Since we also ask for $\beta_1(x,g^1_x)$ to be large, we can apply the first convergence in Corollary \ref{coro:threeBounds} to conclude.

  %
\end{proof}
In the following, we fix some $\epsilon\in (0,\tfrac{1}{20})$. We call $G_R$ the intersection of $F_R$, $E_R$, and the complementaries of the three events appearing in Proposition \ref{prop:threeBounds} (with the chosen $\epsilon$). In particular, for any given $R$, $G_R$ has probability arbitrarily close to $1$ when $R$ is large.

\begin{lemma}
\label{le:card}
As $K\to \infty$,
  \[K^{-\frac{1}{2}-\epsilon} \# \{x\in \mathcal{P}:
  \beta_1(x,g^1_x) > K^{\frac{1}{2}-\epsilon}\} \mbox{ and }K^{-\frac{1}{3}} \# \{x\in \mathcal{P}:
  \beta_1(x,g^1_x) > K^{\frac{2}{3}}\} \]
  converges in probability towards $\frac{1}{\pi}$.
\end{lemma}
\begin{proof}
  Let $\mathcal{P}'$ be the set of points $x$ in $\mathcal{P}$ for  which $S^{(2)}(x,g^1_x)\geq K^{\frac{1}{2}-2 \epsilon}$.
  On the event $G_R$, we know that there is at most $K^{5\epsilon}$ of them. Since $K^{-\frac{1}{2}+\epsilon} K^{5\epsilon}\underset{K\to \infty}{\to} 0$ and $K^{-\frac{1}{3}}K^{5\epsilon} \underset{K\to \infty}{\to} 0$,
  it suffices to show that
  \[K^{-\frac{1}{2}-\epsilon} \# \{x\in \mathcal{P}\setminus \mathcal{P}':
  \beta_1(x,g^1_x) > K^{\frac{1}{2}-\epsilon}\} \mbox{ and }K^{-\frac{1}{3}} \# \{x\in \mathcal{P}\setminus \mathcal{P}':
  \beta_1(x,g^1_x) > K^{\frac{2}{3}}\} \]
  converges in probability towards $\frac{1}{\pi}$.

  Since $ \theta(x,\bar{X})=\sum_{i\in I_x(g^1_x)} \alpha_i(g^1_x)$, we have
  \begin{align}
  \{x\in \mathcal{P}\setminus \mathcal{P}':
  |\theta(x,\bar{X})| > K^{\frac{1}{2}-\epsilon}+ K^{\frac{1}{2}-2 \epsilon} \}
  & \subseteq
  \{x\in \mathcal{P}\setminus \mathcal{P}' :
  \beta_1(x,g^1_x) > K^{\frac{1}{2}-\epsilon}\} \nonumber
  \\ &\subseteq
  \{x\in \mathcal{P}:
  |\theta(x,\bar{X})| > K^{\frac{1}{2}-\epsilon}- K^{\frac{1}{2}-2 \epsilon} \}.
  \label{eq:incadr}
  \end{align}
  Let $D_k$ be the area of the set of points $z\in \R^2$ such that
  $|\theta(x,\bar{X})| >k$. Then $kD_k$ converges towards $\tfrac{1}{\pi}$ in $L^2$ and therefore in probability (see \cite{Werner}).
  It follows that both
  \[K^{-\frac{1}{2}-\epsilon} \#
  \{x\in \mathcal{P}\setminus \mathcal{P}':
  |\theta(x,\bar{X})| > K^{\frac{1}{2}-\epsilon}+ K^{\frac{1}{2}-2\epsilon} \}
  \mbox{ and }
  K^{-\frac{1}{2}-\epsilon} \#
  \{x\in \mathcal{P}:
  |\theta(x,\bar{X})| > K^{\frac{1}{2}-\epsilon}- K^{\frac{1}{2}-2\epsilon} \}
  \]
  converges in probability towards $\tfrac{1}{\pi}$ as $K$ goes to infinity.

  This proves the first convergence. The second is obtain in an identical way.

\end{proof}
We now consider $H_R$ the large probability event
\[
G_H\cap \big\{ K^{-\frac{1}{2}-\epsilon} \# \{x\in \mathcal{P}: \beta_1(x,g^1_x) > K^{\frac{1}{2}-\epsilon}\} \leq 1\big\}\cap \big\{ -K^{\frac{1}{3}} \# \{x\in \mathcal{P}: \beta_1(x,g^1_x) > K^{\frac{2}{3}}\} \leq 1\big\}.
\]

\section{End of the proof}
\label{sec:end}
In order to finally conclude the proof of the main theorem, we split $\mathcal{P}_K^R$ into a disjoint union $\mathcal{P}_K^R= \mathcal{P}^0\sqcup\mathcal{P}^1\sqcup\mathcal{P}^2\sqcup\mathcal{P}^3$, with
\begin{align*}
\mathcal{P}^0 &=\{x \in \mathcal{P}_K^R : \beta_1(x,g^1_x) > K^{\frac{2}{3} }  \},\\
\mathcal{P}^1&=\{x \in \mathcal{P}_K^R :  K^{\frac{1}{2}-\epsilon} <\beta_1(x,g^1_x) \leq K^{\tfrac{2}{3}} \},\\
\mathcal{P}^2&=\{x \in \mathcal{P}_K^R :   \beta_1(x,g^1_x) \leq  K^{\frac{1}{2}-\epsilon} \mbox{ and } S^{(2)}(x,g^1_x) \geq  K^{\frac{1}{2}-\epsilon} \},\\
\mathcal{P}^3&=\{x \in \mathcal{P}_K^R :   \beta_1(x,g^1_x) \leq  K^{\frac{1}{2}-\epsilon} \mbox{ and } S^{(2)} (x,g^1_x) \leq  K^{\frac{1}{2}-\epsilon}  \}.
\end{align*}
The elements of $\mathcal{P}^0$ are somehow the most important, in the sense that there are the ones that are supposed to contribute to the limit. Besides, for each of these points in $\mathcal{P}^0$, there is one particular place where it appears with a large exponent, and this place is the most important.

%
%
%

For $k\in \{1,\dots, \#\mathcal{P}^0\}$, let $x_k$ be the $k^{\mbox{\scriptsize th}}$ element of $\mathcal{P}^0$ (with $\mathcal{P}^0$ inheriting the order of $\mathcal{P}_K$).
Let also $i_k$ be the unique index $i_k$ such that $\alpha_{x_{k},i_k}(g^1_{x_k} )=\beta_1(x_k,g^1_{x_k} )$, and let us write $c_{x_k}\circ\phi_{x_k}([\bar{X}])$ as
\[
c_{x_k}\circ\phi_{x_k}([\bar{X}])=
R^1_k
M_k
R^2_k
\]
with
\[M_k= w_{k,i_k} x_k^{\beta_1(x_k,g^1_{x_k} )}    w_{k,i_k}^{-1},\]
\[
R^1_k= w_{k,1} x_k^{ \alpha_{x_k,1}(g^1_{x_k} ) }w_{k,1}^{-1}\dots
w_{k,i_k-1} x_k^{ \alpha_{x_k,i_k-1}(g^1_{x_k} ) }w_{k,i_k-1}^{-1},\]
\[ R^2_k=w_{k,i_k+1} x_k^{ \alpha_{x_k,i_k-1}(g^1_{x_k} ) }w_{k,i_k+1}^{-1} \dots w_{k,n} x_k^{ \alpha_{x_k,n}(g^1_{x_k} ) }w_{k,n}^{-1},
\]
where the $w_{k,j}$ all lies in $\mathbb{F}_{\mathcal{P}_{x_k}\setminus \{x_k\}}$.
%
%
Finally, let
\[P_k=\Big(\prod_{x_{k-1}<x<x_k} c_x\circ \phi_x([\bar{X}])\Big),\]
with $x_0<x$ and $x<x_{\#\mathcal{P}_0}+1$ being empty conditions,
so that
\[
[\bar{X}]=P_1 (c_{x_1}\circ\phi_{x_1}([\bar{X}])) P_2\dots P_{{\#\mathcal{P}_0}} (c_{x_{\#\mathcal{P}_0} }\circ\phi_{x_{\#\mathcal{P}_0} }([\bar{X}])) P_{{\#\mathcal{P}_0} +1}=
\Big(\prod_{k=1}^{\# \mathcal{P}^0}  P_k\
R^1_k \
M_k\ R^2_k  \Big) P_{{\#\mathcal{P}_0} +1}.
\]

In the following, for $g_K,h_K\in \mathbb{F}_{ \mathcal{P}_K}$, we write $g_K\approx h_K$ if for any compact Lie group $G$, $\omega_K(g_K)$ converges in distribution if and only if
$\omega_K(h_K)$ converges in distribution, and the limit distributions are the same if they exists. We also write $g_K\simeq h_K$ if $\omega_K(g_K h_K^{-1})$ converges in distribution towards $1$, which is a stronger statement.

\begin{proposition}
\label{prop:ish}
\begin{equation}
[\bar{X}]\simeq \Big(\prod_{k=1}^{\# \mathcal{P}^0}  P_k\
M_k\ \Big) P_{{\# \mathcal{P}^0}+1} \approx
\prod_{k=1}^{\# \mathcal{P}^0}x_k^{\beta_1(x,g^1_{x_k})}
\simeq \prod_{k=1}^{\# \mathcal{P}^0}x_k^{\theta(x,\bar{X} )}
\approx \prod_{k=1}^{\# \mathcal{P}_K}x_k^{\theta(x,\bar{X} )}. \label{eq:program}
\end{equation}
\end{proposition}
Remark that this proposition, together with Corollary \ref{coro:4.7} allows to conclude to the main theorem. It would suffice to have $\approx$ relations everywhere, but the kind of arguments that show $\simeq$ relations are very different from the kind of arguments that show $\approx$ relations.

To prove Proposition \ref{prop:ish}, it only lack us two more lemma, in which the compact Lie group $G$ finally appears.

\begin{lemma}
  \label{le:reduc}
  Let $G$ be a compact Lie group endowed with a biinvariant metric. For all $K$, let $n_K$ be a random integer, and let $(X_i)_{i\in \{1,\dots n_k\} }$ and $(Y_i)_{i\in \{1,\dots n_k+1\} }$ be two sequences of $G$-valued random variables .

  Assume that $Y_1\dots Y_{n_K+1}$ converges in distribution towards $1$ (as $K$ goes to infinity), and that
  \[(X_i)_{i\in \{1,\dots,n_K\}}\overset{(d)}=((Y_1\dots Y_i)   X_i (Y_1\dots Y_i)^{-1} )_{i\in \{1,\dots,n_K\}}.
  \]

  Then, $X_1\dots X_{n_K}$ converges in distribution if and only if $Y_1 X_1Y_2\dots X_{n_K}Y_{n_K+1}$ converges in distribution.
\end{lemma}
\begin{proof}
  Remark that
  \begin{equation} Y_1 X_1Y_2\dots X_{n_K}Y_{n_K+1}
  = \Big(\prod_{i=1}^{n_K} ((Y_1\dots Y_i)   X_i (Y_1\dots Y_i)^{-1})\Big) Y_1\dots Y_{n_K}. \label{eq:sansnom}
  \end{equation}
  Since $Y_1\dots Y_{n_K}$ converges to $1$, the left-hand-side of \eqref{eq:sansnom} converges if and only if
  \[\Big(\prod_{i=1}^{n_K} ((Y_1\dots Y_i)   X_i (Y_1\dots Y_i)^{-1})\Big)\] converges as well, and the limits are the same. This last expression is equal in distribution to the product $X_1 \dots X_{n_K}$, hence the lemma.
  %
  %
\end{proof}
\begin{proposition}
\label{prop:tech}
  For a totally ordered finite set $\mathcal{P}$ and $g\in \mathbb{F}_{ \mathcal{P}}$, set $|g|_2\geq 0$ such that
  \[
  |g|_2^2=\sum_{x \in  \mathcal{P}} S^{(1)}(x,\pi^{\leq x}(g) )^2.
  \]

  Let $G$ be a compact Lie group, endowed with a biinvariant metric. Let $(H_i)_{i\in \mathcal{P}}$ a family of i.i.d. $\Ad$-invariant and symmetric $\mathfrak{g}$-valued random variables, with support on the ball of radius $K^{-1}$, and let
  $\omega_K: \mathbb{F}_{ \mathcal{P}} \to G$ be the random group morphism determined by $\omega_K(e_i)=\exp_G( H_i)$.

  There exists $C,c>0$ such that for any totally ordered finite set $\mathcal{P}$ with $\# \mathcal{P}\geq 4$, for all $K \geq 1$,  and all
  $g\in \mathbb{F}_{ \mathcal{P}}$ with
  $ |g|_2^2\leq cK^2$,
  \[
  \mathbb{E}[
  d_G(\omega_K(g),1)] \leq C \big( K^{-1}|g|_2+
    \frac{\log(\# \mathcal{P})^2}{\log (\log (\# \mathcal{P} ) ) } K^{-2} |g|_2^2 \big).
  \]
\end{proposition}
The proof of this is technical, and postponed to the next section.
Assuming it to hold, we now prove Proposition \ref{prop:ish}.
\begin{proof}[Proof of Proposition \ref{prop:ish}]
Let $G$ be a compact Lie group, endowed with a biinvariant metric $d_G$. Since the metric is biinvariant,
\begin{align*}
d_G(\omega_K([\bar{X}]), \omega_K\Big( \Big(\prod_{k=1}^{\# \mathcal{P}^0}  P_k\
M_k\ \Big) P_{k+1} )\Big)
&\leq \sum_{k=1}^{\# \mathcal{P}^0} d_G(P_kR^1_kM_kR^2_k, P_kM_k)\\
&\leq \sum_{k=1}^{\# \mathcal{P}^0} (d_G(\omega_K(R_K^1),1)+d(\omega_K(R_K^2),1))\\
&\leq
\sum_{k=1}^{\# \mathcal{P}^0} \sum_{i\neq i_k} |\alpha_{x_k,i}(g^1_{x_k})|
d_G(\omega_K(x_k),1)\\
&\leq \sum_{x\in \mathcal{P}^0} S^{(2)}(x,g^1_x ) K^{-1}.
\end{align*}
Since $\# \mathcal{P}^0\leq K^{\frac{1}{3}}$ and $S^{(2)}(x,g^1_x)\leq K^{\frac{1}{2}}$ for all $x\in \mathcal{P}^0$, we end up with
\[
d_G\Big(\omega_K([\bar{X}]), \omega_K\Big( \Big(\prod_{k=1}^{\# \mathcal{P}^0}  P_k\
M_k\ \Big) P_{k+1} )\Big)\leq K^{-\frac{1}{6}}\underset{K\to \infty}\longrightarrow 0.
\]
This proves the first relation of \eqref{eq:program}.

We now look at the second one.
Thanks to Lemma \ref{le:reduc}, the problem reduces to show that
\[\omega_K\Big(\prod_{k=1}^{\# \mathcal{P}^0+1} P_k \Big)\underset{K\to \infty}\longrightarrow 0.\]
We want to apply Proposition \ref{prop:tech}, for which we need to bound
\[
|\prod_{k=1}^{\# \mathcal{P}^0+1} P_k|_2^2=\sum_{x\in \mathcal{P}^1\sqcup \mathcal{P}^2\sqcup \mathcal{P}^3} S^{(1)}(x, g^1_x )^2.
\]

For $x\in \mathcal{P}_1$, we know that $\beta^1(x,g^1_x)\leq K^{\frac{2}{3}}$, hence that $S^{(1)}(x, g^1_x)\leq K^{\frac{2}{3}}+K^{\frac{1}{2}+\epsilon}$. Since we also know that
$\# \mathcal{P}^1\leq K^{\frac{1}{2}+\epsilon }$,
\[
  \sum_{x\in \mathcal{P}^1} S^{(1)}(x,g^1_x)^2\leq K^{\frac{1}{2}+\epsilon} ( K^{\frac{2}{3}}+K^{\frac{1}{2}+\epsilon})^2\sim K^{\frac{11}{6}+\epsilon}.
\]

For $x\in \mathcal{P}^2$, we know that $S^{(1)}(x,g^1_x)\leq 2 K^{\frac{1}{2}+\epsilon}$. Since we also know that
$\# \mathcal{P}^2\leq K^{5\epsilon }$,
\[
  \sum_{x\in \mathcal{P}^1} S^{(1)}(x,g^1_x)^2\leq 4 K^{1+7\epsilon}.
\]

For $x\in \mathcal{P}^3$, we know that $S^{(1)}(x, g^1_x)\leq 2 K^{\frac{1}{2}-\epsilon}$. Since we also know that
$\# \mathcal{P}^2\leq \# \mathcal{P} \leq K\log(K) $,
\[
  \sum_{x\in \mathcal{P}^1} S^{(1)}(x, g^1_x)^2\leq 4 \log(K) K^{2-2\epsilon}.
\]

Combining this three results together, we obtain, for $K$ large enough,
\[
|\prod_{k=1}^{\# \mathcal{P}^0+1} P_k|_2^2   \leq  K^{2-\epsilon}.
\]
This is enough to apply Proposition \ref{prop:tech}, and to deduce that
\[
\mathbb{E}\Big[ d_G(\omega_K\Big(
\prod_{k=1}^{\# \mathcal{P}^0+1} P_k\Big),1)\Big]\leq K^{-\frac{\epsilon}{2}}+\log(K)^2 K^{-\epsilon}  \underset{K \to \infty}\longrightarrow 0.
\]
It follows that $\omega_K\Big(
\prod_{k=1}^{\# \mathcal{P}^0+1} P_k\Big)$ converges to $0$ in probability, which proves the second relation.

The two last relations are proved in a similar, but easier, way. First,
\begin{align*}
d_G\Big( \prod_{k=1}^{\# \mathcal{P}^0} x_k^{\beta_1(x,g^1_{x_k})}, \prod_{k=1}^{\# \mathcal{P}^0} x_k^{\theta(x,\bar{X})}  \Big)
&\leq \sum_{k=1}^{\# \mathcal{P}^0} |\theta(x,\bar{X})-\beta_1(x,g^1_{x_k})| d_G(x_k,1)\\
&\leq (\# \mathcal{P}^0) K^{-1} \max_{x\in \mathcal{P}_0}(|\theta(x,\bar{X})-\beta_1(x,g^1_{x_k})| )\\
&\leq K^{\frac{1}{3}} K^{-1} K^{\frac{1}{2}-\epsilon}\underset{K\to \infty}\to 0.
\end{align*}
The last inequality follows from the events $G_R$ (recall the first bound in Proposition \ref{prop:threeBounds}) and $H_R$ (recall Lemma \ref{le:card}). This proves the third relation.

For the fourth one, we apply Lemma \ref{le:reduc} once more, and the problem reduces to show that the product $\prod_{x\in \mathcal{P}_K\setminus \mathcal{P}^0} \omega_K(x)^{\theta(x,\bar{X})}$ converges towards $0$. Since \[\sum_{x\in \mathcal{P}_K\setminus \mathcal{P}^0} \theta(x,\bar{X})^2\leq \sum_{x\in \mathcal{P}_K\setminus \mathcal{P}^0} S^{(1)}(x,\bar{X})^2\leq K^{2-\epsilon},\]
we can use Proposition \ref{prop:tech} to conclude. Remark that the situation here is actually much easier than for the second relation: the letters here appear on the right order (so that we could avoid the use of Proposition \ref{prop:tech}).
\end{proof}
This concludes the proof of the main theorem, up to the proof of Proposition \ref{prop:tech} that we postponed.

\section{An inequality for words in random matrices}
\label{sec:tech}

This section is devoted to the proof of the last piece missing, Proposition \ref{prop:tech}. Let us first recall that
 \[
  |g|_2^2=\sum_{x \in  \mathcal{P}} S^{(1)}(x,\pi^{\leq x}(g) )^2.
 \]
 and that Proposition \ref{prop:tech} states the following.
\begin{proposition*}
  Let $G$ be a compact Lie group, endowed with a biinvariant metric. Let $(H_i)_{i\in \mathcal{P}}$ a family of i.i.d. $\Ad$-invariant and symmetric $\mathfrak{g}$-valued random variables, with support on the ball of radius $K^{-1}$, and let
  $\omega_K: \mathbb{F}_{ \mathcal{P}} \to G$ be the random group morphism determined by $\omega_K(e_i)=\exp_G( H_i)$.

  There exists $C,c>0$ such that for any totally ordered finite set $\mathcal{P}$ with $\# \mathcal{P}\geq 4$, for all $K \geq 1$,  and all
  $g\in \mathbb{F}_{ \mathcal{P}}$ with
  $ |g|_2^2\leq cK^2$,
  \[
  \mathbb{E}[
  d_G(\omega_K(g),1)] \leq C \big( K^{-1}|g|_2+
    \frac{\log(\# \mathcal{P})^2}{\log (\log (\# \mathcal{P} ) ) } K^{-2} |g|_2^2 \big).
  \]
\end{proposition*}
\begin{remark}
We think the second term can actually be omitted. This is possible when the group is commutative, and the proposition then follows from a basic analysis of the variance.

When the group is not commutative,
it is possibly simpler to prove that
\[  \mathbb{E}[
  d_G(\omega_K(g),1)] \leq C K^{-1}|g|'_2,
  \]
where ${|g|'}_2^2=\sum_{x\in \mathcal{P}} S^{(1)}(x,g)^2$,
  Such a bound would not be sufficient for our purpose, precisely because we know no bound on $S^{(1)}(x,g)$ but only on $S^{(1)}(x,\pi^{\leq x} (g))$.
\end{remark}

We will need a few times to compare the exponential of sums in $\mathfrak{g}$ with the product in $G$ of the exponentials.

First, let is remark that, from the Dynkin's formula, we get, for $X,Y\in \mathfrak{g}$,
\[\ln(\exp(sX)\exp(tY))\underset{s,t\to 0}{=} sX+tY+\frac{st}{2} [X,Y]+O(st(s+t) ),\]
or stated otherwise,
\[\ln(\exp(X)\exp(Y))-X-Y \underset{\|X\|,\|Y\|\to 0}{=} \frac{1}{2} [X,Y]+O(\|X\|\|Y\|(\|X\|+\|Y\|)),\]
Since $G$ is compact, it follows that there exists a constant $C$ such for all $X,Y\in \mathfrak{g}$,
\[
d_G(\exp_G(X)\exp_G(Y), \exp_G(X+Y))\leq C \|X\| \|Y\|.
\]

\begin{lemma}
\label{le:approx1}
  There exists a constant $C$ such that all positive integer $n$, for all $X_1,\dots, X_n\in \mathfrak{g}$,
  \[
  d_G\Big( \prod_{i=1}^n \exp_G(X_i), \exp_G \big(\sum_{i=1}^n X_i\big)\Big)\leq C \Big(\sum_{i=1}^N \|X_i\| \Big)^2 .
  \]
\end{lemma}
\begin{proof}
  For $j\in \{0,\dots, n\}$, let $Y_j= \sum_{i=1}^j X_i$ and $z_j=\prod_{i=n-j+1}^n \exp_G(X_i)$. We are thus trying to bound $d(z_n,\exp_G(Y_n))$. Remark that
  $z_{n-j}=\exp_G(X_{j+1})z_{n-j-1}$ and $Y_{j+1}=X_{j+1}+Y_{j}$.
  Applying the triangle inequality between the points $z_n=\exp_G(Y_1) z_{n-1}, \dots, \exp_G(Y_{n-1}) z_{1}, \exp_G(Y_n)$, we obtain
  \begin{align*}
  d(z_n,\exp_G(Y_n))&\leq \sum_{j=1}^{n-1} d_G\big(\exp_G(Y_j) z_{n-j} ,\exp_G(Y_{j+1} ) z_{n-j-1}     \big)\\&= \sum_{j=1}^{n-1} d_G\big(\exp_G(Y_j)\exp_G(X_{j+1})z_{n-j-1} ,\exp_G(Y_{j+1} )   z_{n-j-1}    \big)\\
  &=\sum_{j=1}^{n-1} d_G\big(\exp_G(Y_j)\exp_G(X_{j+1}),\exp_G(Y_{j}+X_{j+1} )   \big) \\ &\leq
  \sum_{1\leq i<j \leq n} C \|X_i\|\|X_{j}\|\leq \frac{C}{2} \big(\sum_{j=1}^{n}\|X_{j}\|\big)^2.
  \end{align*}
\end{proof}

Under the assumptions on $g$ in Proposition \ref{prop:tech}, for all $x\in \mathcal{P}$, $d_G(\omega_K(c_x(\phi_x(g)),1)<c$, so that the logarithm of $\omega_K(c_x(\phi_x(g))$ is well-defined, as an element of $\mathfrak{g}$, provided that $c$ is small enough. During this section, we will call this logarithm $X_x(g)$, so that $\exp_G(X_x(g))=\omega_K(c_x(\phi_x(g))$.

\begin{lemma}
\label{le:approxlie}
  Let $G$ be a compact Lie group, endowed with a biinvariant metric. Let $(H_x)_{x\in \mathcal{P}}$ a family of $\Ad$-invariant and symmetric $\mathfrak{g}$-valued random variables, with support on the ball of radius $K^{-1}$, and let
  $\omega_K: \mathbb{F}_{ \mathcal{P}} \to G$ be the random group morphism determined by $\omega_K(e_i)=\exp_G( H_i)$.
  Assume further that for all $x\in \mathcal{P}$, $H_x$ is independent from $(H_y)_{y<x}$ conditionally on the conjugacy class ${\overline{H}_x}$ of $H_x$ in $G$.

  There exists $c>0$ and a constant $C$ such that for any totally ordered finite set $\mathcal{P}$ with $\# \mathcal{P}\geq 16$, for all $K \geq 1$,  and all
  $g\in \mathbb{F}_{ \mathcal{P}}$ with
  $  |g|_2\leq c  K $,

  \[
  \mathbb{E}\Big[  \Big\| \sum_{x\in \mathcal{P}} X_x(g) \Big\|^2 \Big] \leq  C K^{-2} |g|_2^2 .
  \]
\end{lemma}
\begin{proof}
  Let us recall that $c_x\circ \phi_x(g)$ is a product of $S^{(1)}(x,g)$ commutators of $x$ and $x^{-1}$:
  \[ c_x \phi_x(g)=\prod_{k=1}^{S^{(1)}(x,g) } w_{x,k} x^{\epsilon_{x,k}} w_{x,k},\]
   with $\epsilon_{x,k}\in \{-1,1\}$ and $w_{x,k}\in \mathbb{F}_{{ \{y:y<x\} }}$.
  Set
  \[ \tilde{X}_x(g)= \sum_k \Ad{\omega_K(w_{x,k}) }( H_x^{\epsilon_{x,k}} ).\]
  Lemma \ref{le:approx1} ensures that
  \[ d_G( \exp_G(X_x(g)), \exp_G(\tilde{X}_x(g)))\leq C K^{-2}S^{(1)}(x,\pi^{\leq x}(g))^2.\]
Since $X_x(g)$ and $\tilde{X}_x(g)$ are both smaller than $K^{-1}S^{(1)}(x,\pi^{\leq x}(g))$, hence smaller than $c$ that can be chosen arbitrary small, and since the Jacobian determinant of the exponential at $0$ is $1$, we can assume
  \[ \|X_x(g)-\tilde{X}_x(g)\|\leq 2 d_G( \exp_G(X_x(g)), \exp_G(\tilde{X}_x(g)))\leq C'K^{-2}S^{(1)}(x,\pi^{\leq x}(g))^2.\]

  From triangle inequalities,
  \begin{align}
  \mathbb{E}\Big[ \big\| \sum_{i} X_x(g) \big\|^2 \Big]
  &\leq 2 \mathbb{E}\Big[ \big\| \sum_{x} \tilde{X}_x(g) \big\|^2 \Big]+ 2\mathbb{E}\Big[  \big(\sum_{x} \big\|X_x(g)-\tilde{X}_x(g) \big\|\big)^2 \Big] \nonumber\\
  &\leq 2\mathbb{E}\Big[ \Big\| \sum_{x} \tilde{X}_x(g) \Big\|^2 \Big]+  C K^{-4} |g|_2^{4}. \nonumber 
  \end{align}

  In order to control the remaining term, let us remark that, since $H_x$ is symmetric and since it is independent from the collection of variables $( (\omega_K(w_{x,k}))_k, (\tilde{X}_y)_{y<x} )$ conditionally on $\overline{H^x}$,
  \[ \mathbb{E}\big[\Ad_{ w_{x,k}  }(H_x) \big|   (\omega_K(w_{x,l}))_{l<k}, (\tilde{X}_y(g))_{y<x}  \big]=\mathbb{E}[H_x]=0.\]
  Therefore, $\mathbb{E}\big[\tilde{X}_x(g) \big|(\tilde{X}_y(g))_{y<x}\big]=0$ and the finite sequence $(S_x= \sum_{y<x} \tilde{X}_y)_{x\in \mathcal{P}}$ is a martingale.
  Let $(S^\alpha_x)_{\alpha}$ be the components of $S_x$ in some orthonormal basis $(e_\alpha)_{\alpha}$ of $\mathfrak{g}$. Then, for each $\alpha$, $(S^\alpha_x)_x$
  is a martingale as well, whose $i^{\mbox{\scriptsize th}}$ steps $\tilde{X}^\alpha_{x_i}$ is supported on \[[-K^{-1} S^{(1)}(x_i,\pi^{\leq x_i}(g) ), K^{-1}S^{(1)}(x_i,\pi^{\leq x_i}(g))].\]
  Thus
  \[
  \mathbb{E}[ (S^\alpha_x)^2 ]= \sum_{y<x} \mathbb{E}[  (\tilde{X}^\alpha_x)^2  ] \leq C K^{-2} |g|_2^2 .
  \]
We sum over $\alpha\in \{1,\dots, \dim(\mathfrak{g}) \}$ to conclude.
\end{proof}

We now intend to prove Proposition \ref{prop:tech}. We invite the reader to look at the similarity between Lemma \ref{le:approxlie} that we have just proved and  Proposition \ref{prop:tech}: they differ by the fact that the former take place in the Lie algebra whilst the latter take place in the Lie group. Some logarithmic (in the size of $\mathcal{P}$) corrections also appears in the Proposition. We think the same result does hold without these correction, but we were not able to prove it.
\begin{proof}[Proof of Proposition \ref{prop:tech}]

Let us first have a naive approach of the problem. We want to show that $\mathbb{E}[d_G(\omega_K(g),1)^2]$ is small. From Lemma \ref{le:approxlie}, we already know that
\[
\mathbb{E}\Big[ d_G\big(\exp_G\big(\sum_{x\in \mathcal{P}} X_x(g)\big), 1\big)^2 \Big]=\mathbb{E}\Big[ \big\|\sum_{x\in \mathcal{P}} X_x(g)\big\|^2 \Big] \leq CK^{-2}|g|^2_2.
\]
Besides, Lemma \ref{le:approx1} gives
\begin{equation}
\label{eq:temp4}
d_G\Big(\omega_K(g),\exp_G\big(\sum_{x\in \mathcal{P}} X_x(g) \big)\Big)\leq C \Big(\sum_{x\in \mathcal{P}} \| X_x(g)\| \Big)^2.
\end{equation}
The main problem is that we would like the square to be \emph{under} the sum, and the best we can get seems to be\footnote{Actually, $\# \mathcal{P}$ can easily be replaced with $\sqrt{\# \mathcal{P}}$ in Equation \eqref{eq:temp5}. It suffices to remark that we can replace
$\Big(\sum_{x\in \mathcal{P}} \| X_x(g)\| \Big)^2$ with $\Big(\sum_{x\in \mathcal{P}} \| X_x(g)\| \Big) \max_x \sum_{y<x} \| X_y(g)\| $ in  Lemma \ref{le:approxlie}, and get a maximal inequality in  Lemma \ref{le:approx1}.
}
\begin{equation}
\label{eq:temp5}
\mathbb{E}\Big[\Big(\sum_{x\in \mathcal{P}} \| X_x(g)\| \Big)^2\Big]\leq (\# \mathcal{P}) \sum_{x\in \mathcal{P}} \mathbb{E}\Big[\| X_x(g)\|^2\Big].
\end{equation}


We will now try to reduce this factor $(\# \mathcal{P})$.
To this end, we use a \emph{divide and conquer} algorithm.
We split the alphabet
$\mathcal{P}$ into $M$ subsets of size $\#\mathcal{P}/M$. The choice $M\simeq \log(\#\mathcal{P})$ gives a nearly optimal result. Each of the subsets is then recursively split in a similar way, until each subsets contains a single element.

Therefore, we consider a $\lceil \log(\#\mathcal{P}) \rceil$-regular rooted tree $\mathbb{T}$ of depth $D= \lceil \frac{ \log(\#\mathcal{P})}{\log(\log (\#\mathcal{P} ))}\rceil$  (see Figure \ref{fig:tree} below).
\begin{figure}[h]
\includegraphics[scale=0.8]{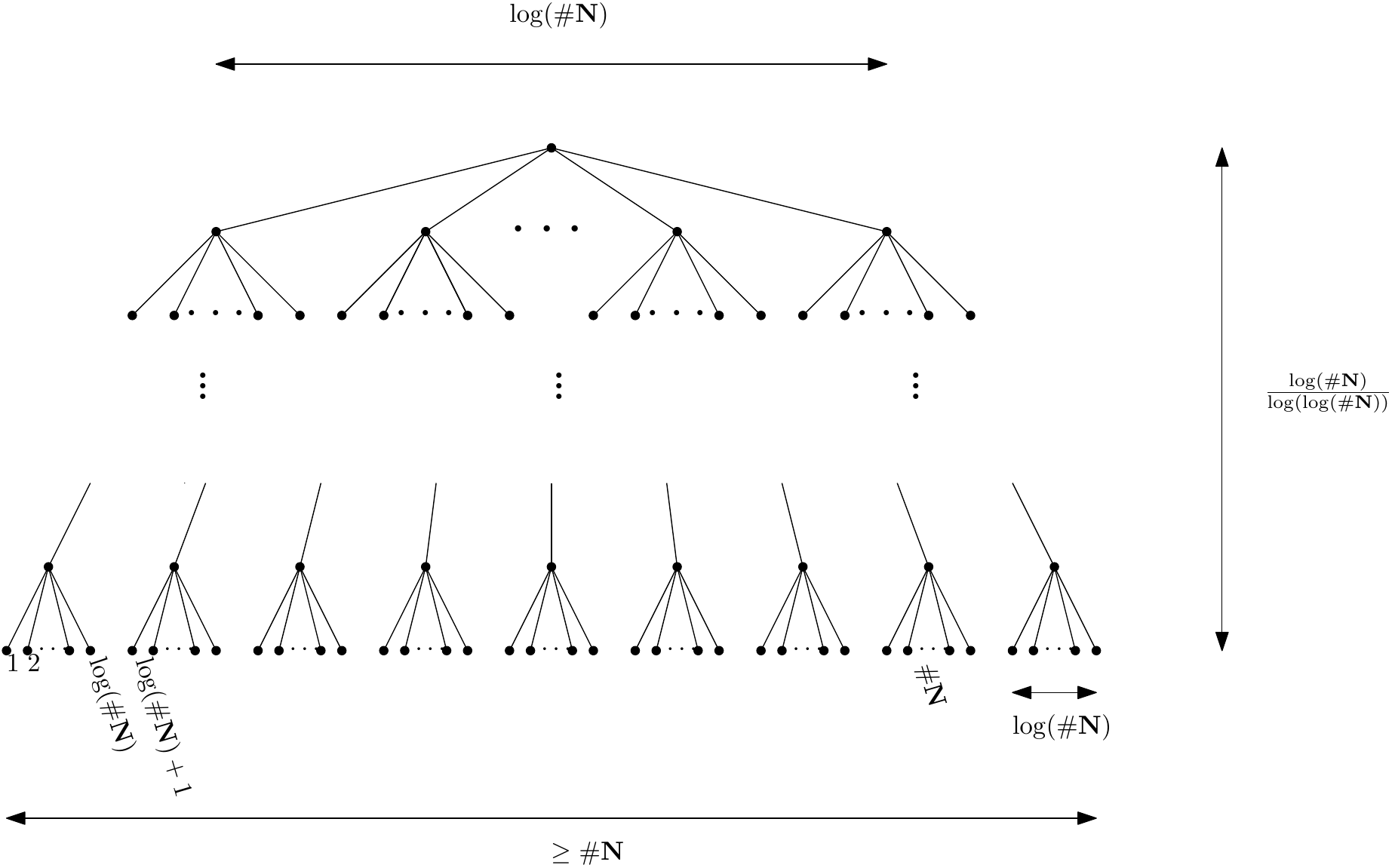}
\caption{\label{fig:tree} The tree $\mathbb{T}$.}
\end{figure}
Since $\#\mathcal{P}=\log(\#\mathcal{P})^{\frac{\log(\#\mathcal{P})}{\log(\log(\#\mathcal{P}))}}\leq \lceil \log(\#\mathcal{P}) \rceil^D$, the tree $\mathbb{T}$ has more than $\#\mathcal{P}$ leafs. We
fix some depth-first traversal of $\mathbb{T}$, and we enumerate the leaf accordingly .
We label the $i^{\mbox{\scriptsize th}}$ leafs by $X_{x_i}$ if $i\leq  \# \mathcal{P}$, and by $0$ otherwise. Then, we inductively label all the internal vertices, from the bottom to the root by the sum of the labels of all its children:
for any internal vertex $v$, setting $\sigma(v)$ the set of the children of $v$ and $l(v)$ the set of leafs descendant from $v$ ,
\[X_v=\sum_{w\in \sigma(i)} X_w=\sum_{l\in l(i)} X_l.\]
 In particular, the root $r$ is labelled by $X_r=\sum_{x\in \mathcal{P}} X_x$.


For $k\in \{1,\dots ,D \}$, we define $\mathbb{T}_k=\{ v: d_{\mathbb{T}}(v,r)=k-1\}$ the set of vertices at depth exactly $k$, which is naturally ordered,
and $g^k= \prod_{v\in \mathbb{T}_k} \exp_G(X_v) $. In particular, $g^1=\exp_G\big( \sum_{x\in \mathcal{P}} X_x\big)$ and $g^D= \exp_G(X_{x_1})\dots \exp_G(X_{x_{\# \mathcal{P}}} )$.

We already know with Lemma \ref{le:approxlie} that $g^1$ is small,
\[ \mathbb{E}[d(g^1,1)^2]\leq \mathbb{E}[\|\sum_{x\in \mathcal{P}} X_x \|^2 ]\leq C K^{-2} |g|_2^2 ,\]
and our goal is to show that $g^D$ is small as well. It thus suffices to show that $d_G(g^i,g^{i+1})$ is small for all $i$.

Since $d_G$ is biinvariant, it satisfies the property that $d_G(ab,cd)\leq d_G(a,c)+d_G(b,d)$ for any $a,b,c,d\in G$, and it follows that
\begin{align*}
 d_G( g^i,g^{i+1})
& = d_{G}\Big( \prod_{v\in \mathbb{T}_i} \exp_G( \sum_{w\in \sigma(v)} X_w), \prod_{v\in \mathbb{T}_i} \prod_{w\in \sigma(v)}  \exp_G(X_w) \Big)\\
&\leq \sum_{ v\in \mathbb{T}_i } d_G\big( \exp_G( \sum_{w\in \sigma(v)} X_w),   \prod_{w\in \sigma(v)}  \exp_G(X_w)  \big)\\
&\leq C \sum_{ v\in \mathbb{T}_i }  \Big( \sum_{w\in \sigma(v)}   \|X_w\|  \Big)^2\quad \mbox{(using Lemma \ref{le:approx1}).}\\
&\leq C (\#\sigma(v))\sum_{ w\in \mathbb{T}_{i+1} }  \|X_w\|^2.
\shortintertext{Therefore,}
\mathbb{E}\big[
 d_G( g^i,g^{i+1})\big]&\leq C \lceil \log (\# \mathcal{P})\rceil \sum_{ w\in \mathbb{T}_{i+1} } \mathbb{E}\big[ \|X_w\|^2\big]
\end{align*}
We now apply Lemma \ref{le:approxlie} to $X_w=\sum_{i\in l(w)} X_i$, and we get
\[ \mathbb{E} \big[\|X_w\|^2 \big]\leq CK^{-2} \sum_{x\in l(w)} S^{(1)}(x,\pi^{\leq x}(g))^2,\]
and therefore
\[\sum_{ w\in \mathbb{T}_{i+1} }  \mathbb{E}[\|X_w\|^2]\leq  CK^{-2}\sum_{ w\in \mathbb{T}_{i+1} }  \sum_{x\in l(w)} S^{(1)}(x,\pi^{\leq x}(g))^2=CK^{-2}|g|_2^2.\]
%
%

We end up with \begin{align*}
\mathbb{E}[\d_G(g^D,1) ]
&\leq\mathbb{E}[ d_G(g^1,1)^2]^{\tfrac{1}{2}}+\sum_{k=1}^{D-1} \mathbb{E}[d_G(g^i,g^{i+1})]\\
&\leq  CK^{-1}|g|_2 +  C D \lceil \log (\# \mathcal{P})\rceil K^{-2}|g|_2^2\\
&\leq C K^{-1}|g|_2+ C' \frac{\log (\# \mathcal{P})^2}{\log(\log(\# \mathcal{P}))}K^{-2}|g|_2^2,
\end{align*}
as announced.
\end{proof}

\section{Acknowledgements}
The author is extremely grateful to Thierry Lévy and to Pierre Perruchaud for their substantial help.
I am pleased to acknowledge support from the ENS Lyon and the LPSM during the early stages of this work, and from the
ERC Advanced Grant 740900 (LogCorRM) during its latter stages.

\newpage

\printbibliography

\end{document}